\documentclass{amsart}
\usepackage{amsmath,amsthm,amsfonts,amssymb,latexsym,mathrsfs,graphicx}
\usepackage{pgf,tikz}
\usepackage{mathrsfs}
\usetikzlibrary{arrows}

\usepackage{hyperref}

\usepackage{enumerate}
\usepackage[shortlabels]{enumitem}
\usepackage{color}

\usepackage[latin1]{inputenc}
\usepackage{bm} 

\usepackage{pstricks-add}
\usepackage{graphicx}
\usepackage{pgf,tikz,pgfplots}
\usepackage{mathrsfs}
\usetikzlibrary{arrows}
\definecolor{ududff}{rgb}{0.30196078431372547,0.30196078431372547,1}

\newtheorem{theorem}{Theorem}[section]

\newtheorem{lemma}[theorem]{Lemma}
\newtheorem{proposition}[theorem]{Proposition}
\newtheorem*{ThmA}{Theorem A}
\newtheorem*{ThmB}{Theorem B}
\newtheorem*{ThmC}{Theorem C}

\newtheorem*{CorD}{Corollary D}
\theoremstyle{definition}

\newtheorem{rem}[theorem]{Remark}

\newenvironment{enumeratei}{\begin{enumerate}[\upshape (a)]}
    {\end{enumerate}}

\def\irr#1{{\rm Irr}(#1)}
\def\cent#1#2{{\bf C}_{#1}(#2)}

\def\syl#1#2{{\rm Syl}_#1(#2)}

\def\nor{\trianglelefteq\,}
\def\oh#1#2{{\bf O}_{#1}(#2)}
\def\zent#1{{\bf Z}(#1)}
\def\V#1{{V}(#1)}
\def\sbs{\subseteq}

\def\fit#1{{\bf F}(#1)}
\def\frat#1{{\bf \Phi}(#1)}

\newcommand{\F}{{\mathbb F}}

\def\irr#1{{\rm Irr}(#1)}

\def\cent#1#2{{\bf C}_{#1}(#2)}

\def\syl#1#2{{\rm Syl}_#1(#2)}
\def\nor{\trianglelefteq\,}
\def\norm#1#2{{\bf N}_{#1}(#2)}
\def\oh#1#2{{\bf O}_{#1}(#2)}

\def\zent#1{{\bf Z}(#1)}
\def\sbs{\subseteq}

\def\fit#1{{\bf F}(#1)}
\def\lay#1{{\bf E}(#1)}
\def\fitg#1{{\bf F^*}(#1)}

\def\SL#1{{\rm SL}_{2}(#1)}
\def\PSL#1{{\rm PSL}_{2}(#1)}

\def\V#1{{\rm V}(#1)}
\def\E#1{{\rm E}(#1)}

\def\irr#1{{\rm Irr}(#1)}

\def\cd#1{{\rm cd}(#1)}

\def\cent#1#2{{\bf C}_{#1}(#2)}
\def\syl#1#2{{\rm Syl}_#1(#2)}
\def\oh#1#2{{\bf O}_{#1}(#2)}

\def\zent#1{{\bf Z}(#1)}

\def\ker#1{{\rm ker}(#1)}
\def\norm#1#2{{\bf N}_{#1}(#2)}

\def\sbs{\subseteq}

\def \o#1{\overline{#1}}
\mathchardef\coso="2023

\def \nq{\mathcal{N}_q}

\begin{document}

\title[Degree graphs with a cut-vertex. I]{Non-solvable groups whose character degree graph has a cut-vertex. I}

\author[]{Silvio Dolfi}
\address{Silvio Dolfi, Dipartimento di Matematica e Informatica U. Dini,\newline
Universit\`a degli Studi di Firenze, viale Morgagni 67/a,
50134 Firenze, Italy.}
\email{silvio.dolfi@unifi.it}

\author[]{Emanuele Pacifici}
\address{Emanuele Pacifici, Dipartimento di Matematica e Informatica U. Dini,\newline
Universit\`a degli Studi di Firenze, viale Morgagni 67/a,
50134 Firenze, Italy.}
\email{emanuele.pacifici@unifi.it}

\author[]{Lucia Sanus}
\address{Lucia Sanus, Departament de Matem\`atiques, Facultat de Matem\`atiques, \newline Universitat de Valencia, Burjassot,
46100 Valencia, Spain.}
\email{lucia.sanus@uv.es}

\author[]{Victor Sotomayor}
\address{Victor Sotomayor, Departamento de Matematicas y Estadistica,\newline
Centro Universitario EDEM, La Marina de Valencia s/n,
46024 Valencia, Spain.}
\email{vsotomayor@edem.es}

\thanks{The research of the first and second authors is partially supported by INDAM-GNSAGA. The research of the third and fourth authors is supported by Proyecto AICO/2020/298, Generalitat Valenciana (Spain). The third author is also partially supported by the Spanish Ministerio de Ciencia e Innovaci\'on PID2019-103854GB-I00.
The fourth author is also supported by Proyecto PGC2018-096872-B-I00 (MCIU/AEI/FEDER, UE), and he acknowledges the financial support of Centro Universitario EDEM (Valencia). This research has been carried out during a visit of the third and fourth authors at the Dipartimento di Matematica e Informatica ``Ulisse Dini" (DIMAI) of Universit\`a degli Studi di Firenze. They thank DIMAI for the hospitality.}


\begin{abstract} 
Let $G$ be a finite group. Denoting by $\cd{G}$ the set of degrees of the irreducible complex characters of $G$, we consider the \emph{character degree graph} of \(G\): this is the (simple undirected) graph whose vertices are the prime divisors of the numbers in $\cd{G}$, and two distinct vertices $p$,  $q$ are adjacent if and only if $pq$ divides some number in $\cd{G}$. In the series of three papers starting with the present one, we analyze the structure of the finite non-solvable groups whose character degree graph possesses a \emph{cut-vertex}, i.e. a vertex whose removal increases the number of connected components of the graph.
\end{abstract}

\maketitle

\section{Introduction}
Given a finite group $G$, let $\cd{G} = \{ \chi(1) \mid \chi \in \irr{G} \}$ be the set of the degrees of the irreducible
complex characters of $G$. This  set has been extensively studied, on its own account as well as in relation to the structure of the group $G$.
The \emph{character degree graph} (or \emph{degree graph}) $\Delta(G)$ is the simple undirected graph related to the set
$\cd{G}$  as follows: the vertex set $\V G$ of $\Delta(G)$ consists of all the prime numbers dividing the degree of some irreducible character of
$G$, and two distinct vertices \(p\), \(q\) are adjacent in $\Delta(G)$ if and only if the product $pq$ divides some irreducible character degree of $G$.
Several results in the literature illustrate that  graph-theoretical features of $\Delta(G)$ are significantly linked to
the structure of $G$. For a detailed account on this subject, we refer  to the expository paper~\cite{L_survey}. 

As an important example, we mention that the number of connected components of $\Delta(G)$ is in general at most three
(\cite{MSW}), and at most two if $G$ is solvable.  
The structure of the finite groups whose degree graph is disconnected is described in~\cite{L} for the solvable case, and 
in \cite{LW} for the non-solvable case.

As a further step in the study of the connectivity properties of the degree graph, it is natural to  consider the question about the existence
of a \emph{cut-vertex}, i.e. a vertex whose removal yields a resulting graph with more connected components than the original one. A graph that is connected and has a cut-vertex is said to have   \emph{connectivity degree $1$}, which is  the smallest degree of connectivity of a connected graph. 
The finite solvable groups $G$ such that $\Delta(G)$ has connectivity degree $1$ are investigated in \cite{LM}; among other things,
it is shown (\cite[Theorem 1.1]{LM}) that in this case $\Delta(G)$ has a unique cut-vertex. 

We also mention that the finite groups \(G\) whose degree graph \(\Delta(G)\) has connectivity degree $1$, and which satisfy the strong assumption (studied in  \cite{ACDPS}, and implying non-solvability)
that the complement graph of $\Delta(G)$ is non-bipartite, are described in \cite{EKM}. 

In the series of three papers starting with the present one, we give a  complete classification of all finite \emph{non-solvable} groups whose degree graph has a cut-vertex
(see Theorems A, B and C in Section~2). It turns out that the structure of the relevant groups, as well as of the corresponding degree graphs,
is significantly restricted and, quite surprisingly, does not fall too far from  the structure of the finite non-solvable groups
with a disconnected degree graph (see \cite{LW} or Theorem~\ref{LewisWhite}).
As an  interesting feature, we mention that in all cases (similarly to the situation for solvable groups) a 
degree graph $\Delta(G)$ with connectivity degree $1$ has a unique cut-vertex (Corollary~D). 

The structure of the finite non-solvable groups whose character degree graph has connectivity degree $1$ can be generally described as
follows. Such a group $G$ has a unique non-solvable composition factor $S$, belonging to the following short list of isomorphism types:
${\rm{PSL}}_2(t^a)$, ${\rm{Sz}}(2^a)$, ${\rm{PSL}}_3(4)$,   ${\rm{M}}_{11}$, ${\rm{J}}_1$; moreover, if  
$S \not\cong {\rm{PSL}}_2(t^a)$, then  $G$ has a (minimal) normal subgroup isomorphic to $S$ (Theorem~\ref{main}). Finally, denoting by $R$ the solvable radical of $G$, the vertex set of $\Delta(G)$ consists exactly of the  prime divisors of the
order of the almost-simple group $G/R$ and, if not already there, the cut-vertex of $\Delta(G)$ (Theorem~\ref{0.2}).  
We also remark that in all cases, except  possibly when the non-abelian  simple section $S$ is isomorphic to the Janko group ${\rm{J}}_1$,
the cut-vertex $p$ of $\Delta(G)$ is a complete vertex (i.e. it is adjacent to all other vertices) of $\Delta(G)$; moreover, the graph
$\Delta(G) - p$ obtained from $\Delta(G)$ by removing the vertex $p$ (and all the edges incident to $p$) has exactly two
connected components, which are complete graphs. Rather remarkably, one of them consists of a single vertex, and hence our
results give also a description of the finite non-solvable groups $G$ such that $\Delta(G)$ has a \emph{cut-edge}, i.e. an edge whose removal increases the number of connected components of the graph.  

In Section~\ref{results} of this paper we present the complete statement of our main results Theorem~A,  Theorem~B and Theorem~C, as well as a discussion concerning the degree graphs associated to the relevant groups with some figures and comments. After another preliminary section, 
in Section~\ref{proofmain} we  prove Theorem~A  under the assumption that the non-abelian simple section $S$ of the group is not isomorphic to a projective special linear group ${\rm PSL}_2(t^a)$.
The case $S \cong {\rm PSL}_2(t^a)$  requires a much longer and rather technical analysis, which will be carried out in the
forthcoming papers~\cite{DPS2} and~\cite{DPS3}, where the proof of Theorem~A and Theorem~B will be completed. 
Finally, in Section~\ref{proofC} we prove Theorem~C, which provides the  classification the finite non-solvable groups $G$ such that $\Delta(G)$
is not connected and has a cut-vertex, and that is rather directly derived from the classification in~\cite{LW}.  

We close this introductory section by recalling that results about the degree graph of finite groups often have a counterpart concerning the prime graph built on the set of \emph{conjugacy class sizes} of groups. The analogue of the problem here considered is studied in \cite{DPSV}, where a classification of the finite groups whose prime graph on conjugacy class sizes has a cut-vertex is provided. However, the situation in that context turns out to be significantly different from the present one, since the relevant groups turn out to be solvable of Fitting heigth at most three and, moreover, there are cases in which the relevant graphs have two cut-vertices.

Not surprisingly, in this paper we use results that depend on the classification of the finite simple groups. Every group considered in the following discussion is tacitly assumed to be a finite group.

\section{The classification}\label{results}
Given a group $G$, we denote by $R = R(G)$ the \emph{solvable radical} (i.e. the largest solvable normal subgroup), and
by $K = K(G)$ the \emph{solvable residual} (i.e. the smallest normal subgroup with a solvable factor group) of $G$. Equivalently, $K(G)$ is the last term of the derived series of $G$. Also, as customary, $\pi(n)$ denotes the set of all prime divisors of a positive integer $n$, and we write $\pi(G)$ for $\pi(|G|)$.

This section is entirely dedicated to presenting the statement of the main theorems (Theorem~A, Theorem~B and Theorem~C), the proof of which begins in the present paper as specified in the Introduction. The first two theorems treat the connected case, and can be regarded as two parts of the same result (in fact, they are stated separately only to improve clarity of the exposition and readability): the former deals with groups with no composition factors isomorphic to \(A_5\), whereas the latter covers the complementary situation. On the other hand, the disconnected case is analyzed in Theorem~C. The rest of the section also presents the associated degree graphs, with some figures and comments.

In order to clarify the statements we also mention that, for $H =  {\rm SL}_2(t^a)$ (where $t$ is a prime),  an $H$-module $V$
over the field \(\mathbb{F}_t\) with \(t\) elements is called \emph{the natural module for $H$} if $V$ is isomorphic to the standard module for ${\rm SL}_2(t^a)$ (or any of its Galois conjugates)  seen as an $\mathbb{F}_t[H]$-module.
We will freely use this terminology also referred to the conjugation action of a group on a suitable elementary abelian normal subgroup.

\begin{ThmA}
  Let $G$ be a non-solvable group with no composition factors isomorphic to \(\SL{4}\cong A_5\), let $R$ and $K$ be, respectively,
 the solvable radical and the solvable residual  of \(G\), and let \(p\) be a prime number.
  Then the graph $\Delta(G)$ is connected and has  cut-vertex $p$,  if and only if $G/R$ is an almost-simple group, $\V G = \pi(G/R)\cup\{p\}$ and, denoting by \(S\) the socle of \(G/R\), one of the following holds.
\begin{enumeratei}
\item {\emph{$S$ is isomorphic to ${\rm{Sz}}(2^a)$.}} In this case, \(K\cong S\) is a minimal normal subgroup of \(G\),
 $a\geq 3$ is a prime, \(p=2^a-1\) and $\V{G/K}\sbs \{p\}$.
\item {\emph{$S$ is isomorphic to ${\rm{PSL}}_3(4)$.}} In this case, \(K\cong S\) is a minimal normal subgroup of \(G\),
  $|G:KR|\in\{1,3\}$, \(p=5\) and $\V{G/K}\sbs \{5\}$.
\item {\emph{$S$ is isomorphic to ${\rm{M}}_{11}$.}} In this case, \(K\cong S\) is a minimal normal subgroup of $G=K\times R$, \(p=5\) and $\V R\sbs\{5\}$.
\item {\emph{$S$ is isomorphic to ${\rm{J}}_1$}}. In this case, \(K\cong S\) is a minimal normal subgroup of $G=K\times R$, \(p=2\) and $\V R\sbs\{2\}$.
\item {\emph{$S$ is isomorphic to $\PSL{t^a}$, where \(t\) is an odd prime and \(t^a>5\).}} In this case $|G/KR|$ is not a multiple of $t$,  $p\neq t$, and one of the following holds.
\begin{enumeratei}
\item[{\rm{(i)}}] \(K\) is isomorphic to \(S\cong\PSL{t^a}\) or to \(\SL{t^a}\), and $\V{G/K}=\{p\}$.
\item[{\rm{(ii)}}] \(K\) contains a minimal normal subgroup \(L\) of \(G\) such that  \(K/L\) is isomorphic to \(\SL{t^a}\) and \(L\) is the natural module for \(K/L\); moreover, $\V{G/K}=\{p\}$.
 \item[{\rm{(iii)}}] \(t^a=13\) and $p=2$. \(K\) contains a minimal normal subgroup \(L\) of \(G\) such that \(K/L\) is isomorphic to \(\SL{13}\), and \(L\) is one of the two \(6\)-dimensional irreducible modules for \(\SL{13}\) over \(\F_3\). Moreover, $\V{G/K}\subseteq\{2\}$.
 \end{enumeratei}
 \item {\emph{$S$ is isomorphic to $\SL{2^a}$, where \(a> 2\)}}, and one of the following holds.
  \begin{enumeratei}
 \item[{\rm{(i)}}] \(K\cong S\) is a minimal normal subgroup of \(G\); also, \emph{either} $\V{G/K}\cup\pi(G/KR)=\{2\}$ and $p=2$, \emph{or} $G/KR$ has odd order, $p\neq 2$ and $\V{G/K}=\{p\}$.
\item[{\rm{(ii)}}] \(K\) contains a minimal normal subgroup \(L\) of \(G\) such that \(K/L\) is isomorphic to \(\SL{2^a}\) and \(L\) is the natural module for \(K/L\); also, $G/KR$ has odd order, $p\neq 2$, $\V{G/K}=\{p\}$ and, for a Sylow $2$-subgroup $T$ of $G$, we have $T'=(T\cap K)'$.
\end{enumeratei}
\end{enumeratei}
In all cases except in case~{\rm{(d)}} when \(R\) is abelian, \(p\) is is a complete vertex of $\Delta(G)$; furthermore, in all cases, $p$ is the unique cut-vertex of \(\Delta(G)\). 
\end{ThmA}

\begin{ThmB}
  Let $G$ be a (non-solvable) group having a composition factor isomorphic to \(\SL{4}\cong A_5\), let $R$ and $K$ be, respectively,
   the solvable radical and the solvable residual  of \(G\), and let \(p\) be a prime number.
  Then the graph $\Delta(G)$ is connected and has a cut-vertex  $p$,  if and only if $G/R$ is an almost-simple group with socle isomorphic to \(\SL{4}\), $\V G = \{2,3,5\}\cup\{p\}$, and one of the following holds.
\begin{enumeratei} 
\item $K$ is isomorphic either to $\SL 4$ or to $\SL 5$, and $\V {G/K} =\{p\}$; if $p=5$, then $K\cong\SL 4$ and $G=K\times R$.
  \item  $K$ contains a minimal normal subgroup \(L\) of $G$ with $|L| = 2^4$. Moreover,  $G = KR$ and
\begin{enumeratei}
 \item[{\rm{(i)}}]either $L$  is  the natural module for \(K/L\), $p \neq 2$ and $\V{G/K}=\{p\}$, 
 \item[{\rm{(ii)}}] or $L$   is isomorphic to  the restriction  to   \(K/L\), embedded  as  $\Omega_4^-(2)$ into
   $\rm{SL}_4(2)$,  of the standard module of $\rm{SL}_4(2)$. Moreover $p = 5$, $G=K\times R_0$ where $R_0=\cent G K$, and  $\V{R_0}=\V{G/K} \subseteq  \{ 5\}$.
\end{enumeratei}
\item $K$ contains a minimal normal subgroup \(L\) of \(G\) such that $K/L$ is isomorphic to $\SL{5}$, and 
\begin{enumeratei}
 \item[{\rm{(i)}}]  either $L $ is the natural module for \(K/L\), $p\not =5$ and $\V {G/K}=\{p\}$,
  \item[{\rm{(ii)}}] or \(L\) is isomorphic to the restriction to  \(K/L\),  embedded in  \({\rm{SL}}_4(3)\), of the standard module of $\rm{SL}_4(3)$,  $p=2$ and $\V {G/K}\sbs \{2\}$.
\end{enumeratei}
\end{enumeratei} 
In all cases, \(p\) is a complete vertex and the unique cut-vertex of \(\Delta(G)\).
\end{ThmB}

Finally, in order to complete the discussion, we provide a characterization of the non-solvable groups whose degree graph is disconnected and has a cut-vertex.

\begin{ThmC}
\label{disconnected_case}
Let \(G\) be a non-solvable group, $R$  the solvable radical of $G$,  $K$ the solvable residual of $G$ and $p$ a prime number.
Then \(\Delta(G)\) is a disconnected graph with cut-vertex $p$ if and only if  $K$ is isomorphic to either $\PSL{t^a}$ or
$\SL{t^a}$ (where \(t\) is a prime and \(t^a\geq 4\)), $G/K$ is abelian,  $|G:KR| = p^b$ with $b \geq 0$, and one of the following holds. 
\begin{enumeratei}
\item  $t$ is odd, $p=2$,  $t^a\not=9$ and $t^a$ is neither a Fermat nor a Mersenne prime. 
\item $t=2$, $t^a \neq 4$,  $p \neq 2$ and $b \geq 1$.  
\end{enumeratei}
Moreover, \(p\) is the unique cut-vertex and a complete vertex of \(\Delta(G)\).
\end{ThmC}


As a consequence of the above results and of Theorem~1.1 of~\cite{LM} we have:

\begin{CorD}
  For every group $G$, the degree graph $\Delta(G)$ has at most one cut-vertex.
\end{CorD}

As regards the (connected) graphs associated with the groups of Theorem~A, they are represented in Table~1 \emph{except for those related to case} (e) \emph{subcase} (i), \emph{and to case} (f) \emph{subcase} (i) \emph{when} $p\neq 2$ (note that the dashed edge represented in Table~1 case (d) is actually in the graph if and only if \(\V R=\{2\}\)). 
For the two missing families of graphs we provide next a description.

Recall that, given two graphs \(\Gamma_1\) and \(\Gamma_2\) with
vertex sets \(\V{\Gamma_1}\) and \(\V{\Gamma_2}\) respectively, the \emph{join} of \(\Gamma_1\) and \(\Gamma_2\) is defined as the graph \(\Gamma_1\ast\Gamma_2\) whose vertex set is 
\(\V{\Gamma_1} \cup\V{\Gamma_2}\), and two vertices are adjacent if and only if either one of them is in \(\V{\Gamma_1}\) and the other one in \(\V{\Gamma_2}\), or they are vertices of the same \(\Gamma_i\) and they are adjacent in \(\Gamma_i\). 

Now, consider a group \(G\) as in case~(e) subcase~(i) of Theorem~A, and adopt the notation of that theorem. Defining \(\Gamma_1\) to be the complete graph whose vertex set is \(\{2\}\cup\pi(G/KR)\), and \(\Gamma_2\) to be the graph consisting of two complete connected components with vertex sets \(\pi(t^a-1)-\V{\Gamma_1}\) and \(\pi(t^a+1)-\V{\Gamma_1}\) respectively, it turns out that \(\Delta(G/R)\) is a graph with two connected components: one of them consists of \(t\) only, whereas the other is \(\Gamma_1\ast\Gamma_2\) (this follows, for instance, from~\cite[Theorem~A]{W1}). Given that, it is not difficult to see that \(\Delta(G)\) is obtained starting from this \(\Delta(G/R)\) and then connecting the vertex \(p\) (both if it is already in \(\pi(G/R)\) or not) with all the other vertices. 

Similarly, if \(G\) is as in case (f) subcase (i) of Theorem~A for $p\neq 2$, one defines \(\Gamma_1\) to be the complete graph whose vertex set is \(\pi(G/KR)\), and \(\Gamma_2\) to be the graph with two complete connected components of vertex sets \(\pi(2^a-1)-\V{\Gamma_1}\) and \(\pi(2^a+1)-\V{\Gamma_1}\) respectively. Then \(\Delta(G/R)\) is as above (with \(2\) playing the role of \(t\), of course) and \(\Delta(G)\) is also obtained in the same way as above. 


\smallskip
Moving to Theorem~B, the graphs arising from it are displayed in Table~2 only in the case when the cut-vertex \(p\) is not already in the set \(\pi(G/R)=\{2,3,5\}\). The missing graphs can be easily described: they are all the paths of length \(2\) with vertex set \(\{2,3,5\}\). Of course there are three of them and they all occur: for instance, all these graphs occur as \(\Delta(G)\) for \(G\) as in case (a) of Theorem~B (it is enough to consider the direct product \(\SL 4\times R\) where \(R\) is a non-abelian \(q\)-group, for \(q\in\{2,3,5\}\)). It is also clear that case (b) subcase (ii) is associated to the path \(2-5-3\), and case (c) subcase (ii) to the path \(3-2-5\). 

Finally, the (disconnected) graphs related to the groups of Theorem~C are displayed in Table~3.

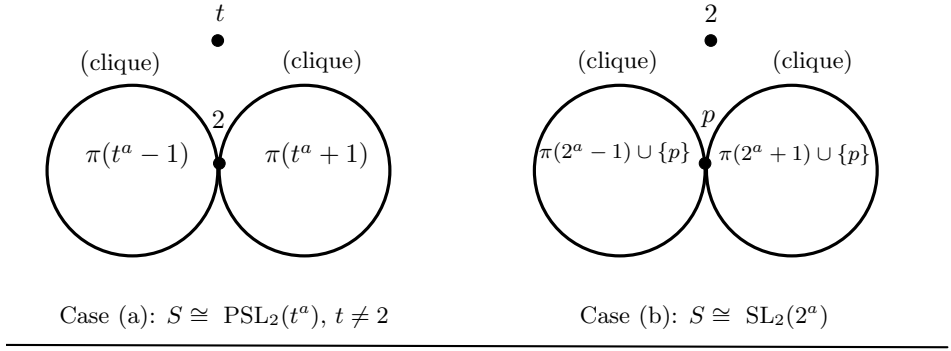
\begin{center}

\begin{tikzpicture}[line cap=round,line join=round,>=triangle 45,x=1.0cm,y=1.0cm]
\clip(-5.670283183081512,-9.037673478219853) rectangle (8.7,4.99);
\draw [line width=1.2pt] (0.6545454545454547,1.994545454545455)-- (2.,2.);
\draw [line width=1.2pt] (3.,3.)-- (2.,2.);
\draw [line width=1.2pt] (2.,2.)-- (3.,1.);
\draw (0.4,1.9) node[anchor=north west] {$2$};
\draw (1.8263049552828954,1.9) node[anchor=north west] {$5$};
\draw (2.833309332078114,3.55) node[anchor=north west] {$3$};
\draw (2.8,1.0) node[anchor=north west] {$7$};
\draw [line width=1.2pt] (4.672727272727274,1.9763636363636368)-- (6.,2.);
\draw [line width=1.2pt] (7.,3.)-- (6.,2.);
\draw [line width=1.2pt] (6.,2.)-- (7.,1.);
\draw (4.4,1.9) node[anchor=north west] {$3$};
\draw (5.809566712383983,1.9) node[anchor=north west] {$5$};
\draw (6.838948964219096,3.584848919639319) node[anchor=north west] {$2$};
\draw (6.8,1.0) node[anchor=north west] {$11$};
\draw [line width=1.2pt] (-3.232,-1.0082)-- (-2.,-2.);
\draw [line width=1.2pt] (-1.0163636363636446,-1.06)-- (-2.,-2.);
\draw [line width=1.2pt] (-2.,-2.)-- (-1.,-3.);
\draw (-3.387739929012349,-2.99) node[anchor=north west] {$5$};
\draw (-2.2,-1.99) node[anchor=north west] {$2$};
\draw [line width=1.2pt] (-1.0163636363636446,-1.06)-- (-1.,-3.);

\draw (-1.1723303000628675,-0.5) node[anchor=north west] {$7$};
\draw (-1.217086050142655,-2.99) node[anchor=north west] {$19$};
\draw (-3.0429841789325616,-3.8236285939245875) node[anchor=north west] {\small Case (d):  $S\cong \text{J}_1$};
\draw (0.3,0.6) node[anchor=north west] {\small Case (b): $S\cong \text{ PSL}_3(4)$};
\draw (4.4,0.6) node[anchor=north west] {\small Case (c): $S\cong \text{M}_{11}$};
\draw [line width=0.8pt] (-4.963430220573079,4.189049680478253)-- (7.709297052154194,4.158140589569161);
\draw [line width=0.8pt] (-4.888084930942077,-0.1857019171304883)-- (7.784642341785204,-0.2166110080395793);
\draw (-4.081454055249056,4.9) node[anchor=north west] {TABLE 1: Graphs arising from Theorem A $\;(\pi=\pi(G/R)\cup \{p\}$)};
\draw [line width=1.2pt] (7.,3.)-- (7.,1.);
\draw [line width=1.2pt] (3.,1.)-- (3.,3.);

\draw [line width=1.2pt] (-3.44,1.94)-- (-2.79,1.9454545454545484);

\draw (-3.8,1.8) node[anchor=north west] {$2$};
\draw (-2.8,1.8) node[anchor=north west] {$p$};

\draw [line width=1.2pt] (-1.67,1.85) circle (1.1402030583337489cm);

\draw (-2.29,3.569180794040912) node[anchor=north west] {\small (clique)};
\draw (-2.29,2.535180794040912) node[anchor=north west] {\small $\pi-\{2\}$};

\draw [line width=0.8pt] (-4.9582,-4.5302)-- (7.3,-4.5611090909090874);
\draw (-3.8,0.6) node[anchor=north west] {\small Case (a): $S\cong \text{ Sz}(2^a)$};
\draw [line width=1.2pt] (-3.232,-2.9882)-- (-2.,-2.);
\draw [line width=1.2pt] (-3.232,-1.0082)-- (-3.232,-2.9882);
\draw [line width=1.2pt] (-1.0163636363636446,-1.06)-- (0.,-2.);
\draw [line width=1.2pt] (-1.,-3.)-- (0.,-2.);
\draw (-3.4324956790921366,-0.5) node[anchor=north west] {$3$};
\draw (0.01,-1.5410853398554154) node[anchor=north west] {$11$};
\draw [line width=1.2pt] (2.208,-2.0482)-- (3.0,-2.0427454545454475);
\draw (2.0,-2.13339859073308) node[anchor=north west] {$t$};
\draw (2.90,-2.13339859073308) node[anchor=north west] {$p$};


\draw (1.99887670280948265,-3.5641817179272414) node[anchor=north west] {\small Case (e)(ii), and Case (f)(ii):};

\draw (1.890083705681833,-3.9788728438447996) node[anchor=north west] {\small $S\cong \text{ PSL}_2(t^a)$, $t^a>5$, $t^a\not =13$};

\draw [line width=0.8pt] (-5.019673455998493,-8.74424507201958)-- (7.653053816728782,-8.775154162928668);
\draw [line width=1.2pt] (-3.7843106117082583,-6.438879183114817)-- (-2.4388560662537127,-6.433424637660272);
\draw [line width=1.2pt] (-1.4388560662537127,-5.433424637660272)-- (-2.4388560662537127,-6.433424637660272);
\draw [line width=1.2pt] (-2.4388560662537127,-6.433424637660272)-- (-1.4388560662537127,-7.433424637660274);
\draw (-3.99,-6.5) node[anchor=north west] {$3$};
\draw (-1.7,-7.49) node[anchor=north west] {$13$};
\draw [line width=1.2pt] (-1.4388560662537127,-7.433424637660274)-- (-1.4388560662537127,-5.433424637660272);
\draw (1.4458810796047015,-14.296474112594906) node[anchor=north west] {$2$};
\draw (-1.6422656759006364,-4.8) node[anchor=north west] {$7$};
\draw (-2.6045143026160678,-6.5) node[anchor=north west] {$2$};
\draw [line width=1.2pt,dash pattern=on 5pt off 5pt] (-2.,-2.)-- (0.,-2.);

\draw (-4.247186805009693,-8.1) node[anchor=north west] {\small Case (e)(iii): $S\cong \text{ PSL}_2(13)$};

\draw (1.6,-6.4) node[anchor=north west] {\footnotesize  $\pi(2^a-1)\cup \{2\}$};
\draw (3.9,-6.4) node[anchor=north west] {\footnotesize  $\pi(2^a+1)\cup \{2\}$};
\draw (3.7,-5.7) node[anchor=north west] {$2$};
\draw [line width=1.2pt] (2.7484751307761712,-6.8172986232771535) circle (1.1348626463258962cm);
\draw [line width=1.2pt] (5.036121399355275,-6.753747596342545) circle (1.1348626463258977cm);

\draw (3.5391512134418255,-0.38) node[anchor=north west] {(\small clique)};
\draw (3.5391512134418255,-1.4) node[anchor=north west] {$\pi$};
\draw [line width=1.2pt] (4.059994740799066,-2.1333232805455307) circle (1.1348626463258986cm);

\draw (1.8768846367697825,-5.1) node[anchor=north west] {(\small clique)};

\draw (4.546800920423833,-5.1) node[anchor=north west] {\small (clique)};

\draw (1.4,-8.1) node[anchor=north west] {\small Case (f)(i): $S\cong \text{ SL}_2(2^a)$, $p=2$};

\begin{scriptsize}
\draw [fill=black] (0.6545454545454547,1.994545454545455) circle (2.2pt);
\draw [fill=black] (2.,2.) circle (2.2pt);
\draw [fill=black] (3.,3.) circle (2.2pt);
\draw [fill=black] (3.,1.) circle (2.2pt);
\draw [fill=black] (4.672727272727274,1.9763636363636368) circle (2.2pt);
\draw [fill=black] (6.,2.) circle (2.2pt);
\draw [fill=black] (7.,3.) circle (2.2pt);
\draw [fill=black] (7.,1.) circle (2.2pt);
\draw [fill=black] (-3.232,-1.0082) circle (2.2pt);
\draw [fill=black] (-2.,-2.) circle (2.2pt);
\draw [fill=black] (-1.0163636363636446,-1.06) circle (2.2pt);
\draw [fill=black] (-1.,-3.) circle (2.2pt);

\draw [fill=black] (-3.54,1.94) circle (2.2pt);

\draw [fill=black] (-2.80,1.9454545454545484) circle (2.2pt);
\draw [fill=black] (-3.232,-2.9882) circle (2.2pt);
\draw [fill=black] (0.,-2.) circle (2.2pt);
\draw [fill=black] (2.208,-2.0482) circle (2.2pt);
\draw [fill=black] (2.93,-2.04466906312302) circle (2.2pt);
\draw [fill=black] (-3.7843106117082583,-6.438879183114817) circle (2.2pt);
\draw [fill=black] (-2.4388560662537127,-6.433424637660272) circle (2.2pt);
\draw [fill=black] (-1.4388560662537127,-5.433424637660272) circle (2.2pt);
\draw [fill=black] (-1.4388560662537127,-7.433424637660274) circle (2.2pt);
\draw [fill=black] (3.8937680831318615,-6.765506140871638) circle (2.2pt);
\end{scriptsize}
\end{tikzpicture}
\end{center}


\begin{center}
\begin{tikzpicture}[line cap=round,line join=round,>=triangle 45,x=1.0cm,y=1.0cm]
\clip(-0.1818181818181825,-1.0963636363636373) rectangle (13.472727272727278,5.303636363636367);
\draw [line width=1.2pt] (0.6545454545454547,1.994545454545455)-- (2.,2.);
\draw [line width=1.2pt] (3.,3.)-- (2.,2.);
\draw [line width=1.2pt] (2.,2.)-- (3.,1.);
\draw (0.48,1.9) node[anchor=north west] {$2$};
\draw (1.8,1.9) node[anchor=north west] {$p$};
\draw (2.836363636363637,3.5) node[anchor=north west] {$3$};
\draw (2.8,1.0) node[anchor=north west] {$5$};
\draw [line width=1.2pt] (4.672727272727274,1.9763636363636368)-- (6.,2.);
\draw [line width=1.2pt] (7.,3.)-- (6.,2.);
\draw [line width=1.2pt] (6.,2.)-- (7.,1.);
\draw (4.49,1.9) node[anchor=north west] {$5$};
\draw (5.74,1.9) node[anchor=north west] {$p$};
\draw (6.836363636363638,3.5) node[anchor=north west] {$3$};
\draw (6.8,1.0) node[anchor=north west] {$2$};
\draw [line width=1.2pt] (8.56286759390931,2.021974094242668)-- (9.908322139363856,2.0274286396972134);
\draw [line width=1.2pt] (10.908322139363856,3.0274286396972134)-- (9.908322139363856,2.0274286396972134);
\draw [line width=1.2pt] (9.908322139363856,2.0274286396972134)-- (10.908322139363856,1.0274286396972132);

\draw (8.4,2.0) node[anchor=north west] {$2$};
\draw (9.7,2.0) node[anchor=north west] {$p$};
\draw (10.7,3.55) node[anchor=north west] {$3$};
\draw (10.7,1.0) node[anchor=north west] {$5$};

\draw (8.92727272727273,0.5945454545454547) node[anchor=north west] {\small Case (b)(i)};
\draw (0.1,0.5945454545454547) node[anchor=north west] {\small Case (a),  $G=\SL 4\times R$};
\draw (4.3,0.5945454545454547) node[anchor=north west] {\small Case (a),  $G\not =\SL 4\times R$};
\draw (5.1,0.2) node[anchor=north west] {\small Case(c)(i)};
\draw [line width=0.8pt] (0.3272727272727267,4.030909090909092)-- (13.,4.);
\draw [line width=0.8pt] (0.35454545454545466,-0.3145454545454545)-- (13.027272727272724,-0.34545454545454546);
\draw (1.2909090909090908,4.7) node[anchor=north west] {TABLE 2: Graphs arising from Theorem B when $p\not \in \{2,3,5\}$};
\draw [line width=1.2pt] (7.,3.)-- (7.,1.);
\draw [line width=1.2pt] (10.908322139363856,3.0274286396972134)-- (10.908322139363856,1.0274286396972132);
\begin{scriptsize}
\draw [fill=black] (0.6545454545454547,1.994545454545455) circle (2.2pt);
\draw [fill=black] (2.,2.) circle (2.2pt);
\draw [fill=black] (3.,3.) circle (2.2pt);
\draw [fill=black] (3.,1.) circle (2.2pt);
\draw [fill=black] (4.672727272727274,1.9763636363636368) circle (2.2pt);
\draw [fill=black] (6.,2.) circle (2.2pt);
\draw [fill=black] (7.,3.) circle (2.2pt);
\draw [fill=black] (7.,1.) circle (2.2pt);

\draw [fill=black] (8.56286759390931,2.021974094242668) circle (2.5pt);
\draw [fill=black] (9.908322139363856,2.0274286396972134) circle (2.5pt);
\draw [fill=black] (10.908322139363856,3.0274286396972134) circle (2.5pt);
\draw [fill=black] (10.908322139363856,1.0274286396972132) circle (2.5pt);


\end{scriptsize}
\end{tikzpicture}
\end{center}


\begin{center}
\begin{tikzpicture}[line cap=round,line join=round,>=triangle 45,x=1.0cm,y=1.0cm]
\clip(-4.8540813442494395,-9.020055864908686) rectangle (8.794900594331786,-3.17598472500909);
\draw (-2.6045266597314396,-3.1) node[anchor=north west] {TABLE 3: Graphs arising from Theorem C};
\draw [line width=0.8pt] (-5.057671754670062,-3.711641226649691)-- (7.615055518057221,-3.742550317558779);
\draw [line width=0.8pt] (-5.019673455998493,-8.74424507201958)-- (7.653053816728782,-8.775154162928668);
\draw (-4.277363797548824,-8.09) node[anchor=north west] {\small Case (a): $S\cong \text{ PSL}_2(t^a)$, $t\not =2$};
\draw (2.6432467628585576,-8.09) node[anchor=north west] {\small Case (b): $S\cong \text{ SL}_2(2^a)$};
\draw (-4.00822894242187,-4.7) node[anchor=north west] {(\small clique)};
\draw (-1.3207608000355636,-4.7) node[anchor=north west] {\small (clique)};
\draw (-2.25,-5.505959238918146) node[anchor=north west] {$2$};
\draw (-3.931333269528455,-5.925005030948703) node[anchor=north west] {$\pi(t^a-1)$};
\draw (-1.547567409832579,-5.925005030948703) node[anchor=north west] {$\pi(t^a+1)$};
\draw (2.60479892641185,-4.7) node[anchor=north west] {\small (clique)};
\draw (5.2,-4.7) node[anchor=north west] {\small (clique)};
\draw (2.1,-5.925005030948703) node[anchor=north west] {\footnotesize  $\pi(2^a-1)\cup \{p\}$};
\draw (4.49,-5.944228949172056) node[anchor=north west] {\footnotesize  $\pi(2^a+1)\cup \{p\}$};
\draw (-2.2011806294266094,-4.1) node[anchor=north west] {$t$};
\draw (4.3,-4.1) node[anchor=north west] {$2$};
\draw (4.265727648290342,-5.505959238918146) node[anchor=north west] {$p$};

\draw [line width=1.2pt] (3.296859982452588,-6.424826904755905) circle (1.1348626463258997cm);

\draw [line width=1.2pt] (5.584506251031695,-6.424826904755905) circle (1.1348626463258997cm);
\draw [line width=1.2pt] (-3.1816004588176554,-6.424826904755905) circle (1.1348626463258997cm);
\draw [line width=1.2pt] (-0.8939541902385486,-6.424826904755905) circle (1.1348626463258997cm);
\begin{scriptsize}

\draw [fill=black] (-2.0473892836397787,-4.694674264654051) circle (2.2pt);

\draw [fill=black] (4.431071157630464,-6.328707313639136) circle (2.2pt);
\draw [fill=black] (4.5079668305238805,-4.694674264654051) circle (2.2pt);

\draw [fill=black] (-2.0262176630488398,-6.328707313639136) circle (2.2pt);
\end{scriptsize}
\end{tikzpicture}

\end{center}


\section{Preliminaries}\label{prel}
A simple undirected graph $\Delta = (V, E)$ is defined by its  vertex set $V = \V{\Delta}$ and its edge set $E = \E{\Delta}$, which is a subset of the set $V^{[2]}$ (consisting of the subsets of size $2$ of $V$).
The graph $\Delta$ is said to be  \emph{complete} if $E = V^{[2]}$. 
When $\Delta = \Delta(G)$ is the degree graph of a group $G$, we will write $\V G$ for $\V{\Delta}$ and $\E G$ for $\E{\Delta}$. 

A graph $\Delta_0 = (V_0, E_0)$ is a \emph{subgraph} of $\Delta = (V, E)$ if $V_0 \sbs V$ and
$E_0 \sbs E \cap V_0^{[2]}$;  $\Delta_0$ is an \emph{induced subgraph} of $\Delta$ if   
$E_0 =  E \cap V_0^{[2]}$; we remark that an induced subgraph of $\Delta$ is uniquely determined by the choice of its vertex set $V_0$.
For $X \subset V$, we denote by $\Delta - X$ the subgraph of $\Delta$ induced by the vertex set $V - X$ (i.e. the subgraph obtained by `deleting' the vertices in $X$ and all the edges incident to them). If $X = \{v\}$, we simply write 
$\Delta -v$. 

A graph $\Delta$ is \emph{connected} if  every pair of  vertices $u$ and $v$ is linked by a path in  $\Delta$.
For $k \geq 0$, $\Delta$ is \emph{$k$-connected} if $|\V{\Delta}| >k$ and $\Delta - X$ is connected for every $X \subseteq V$ such
that $|X| = k$. The largest integer $k$ such that $\Delta$ is $k$-connected is called the \emph{connectivity degree} of
$\Delta$.
So, a disconnected graph has connectivity degree $0$, while a graph $\Delta$ has connectivity degree $1$ if and only if
$|\V{\Delta}| \geq 2$, $\Delta$ is connected, and  there exists a vertex $v \in \V{\Delta}$ such that $\Delta - v$ is disconnected; in this case, $v$ is called a \emph{cut-vertex} of $\Delta$.  
A  vertex $v$ of a graph $\Delta$ is said to be \emph{complete} if $v$ is adjacent to every other vertex of $\Delta$.
It is easily seen that a graph with at least three vertices and  that posseses at least two complete vertices is $2$-connected. A complete subgraph is called a \emph{clique}.
We refer to~\cite{Di} for further background in graph theory.

\smallskip
Throughout the paper some standard  facts of character theory, such as Clifford correspondence, Gallagher's theorem, properties of character extensions  (see \cite{Is}) and coprime actions  (see \cite{A}) will be used without references.

A basic result for the subject at hand, is the following consequence of the Ito-Michler's theorem, that describes the set of the vertices of the degree graph of a finite group.

\begin{proposition}\label{Ito}
Let $G$ be a group and $p$ a prime number. Then $p \not\in \V{G}$ if and only if $G$ has a normal abelian Sylow $p$-subgroup. 
\end{proposition}

We also mention that, as a consequence of Clifford theory, if $H$ is isomorphic either to a normal subgroup or to a factor group
of a group $G$ (even more generally, to a subnormal section of $G$), then $\Delta(H)$ is a subgraph of $\Delta(G)$. 

\smallskip
As an important reference for our discussion, we recall some properties of the degree graphs of the non-abelian simple groups (\cite{W}). 

\begin{lemma}[\mbox{\cite[Theorem~5.2]{W}}]
  \label{PSL2}
Let $S \cong \PSL{t^a}$ or $S \cong \SL{t^a}$, with $t$ prime and $a \geq 1$. 
Let $\pi_{+} = \pi(t^a+1)$ and $\pi_{-} = \pi(t^a-1)$. For a subset $\pi$ of vertices 
of $\Delta(S)$, we denote by $\Delta_{\pi}$ the subgraph of $\Delta = \Delta(S)$ induced
by the subset $\pi$.
Then 
\begin{enumeratei}
\item if $t=2$ and $a \geq 2$, then $\Delta(S)$ has three connected components, $\{t\}$, $\Delta_{\pi_{+}}$ and 
$\Delta_{\pi_{-}}$, and each of them is a complete graph.  
\item if $t > 2$ and $t^a > 5$, then  $\Delta(S)$ has two connected components, 
$\{t\}$ and  $\Delta_{\pi_{+} \cup \pi_{-}}$; moreover, both  $\Delta_{\pi_{+}}$ and $\Delta_{\pi_{-}}$ are
complete graphs, no vertex  in $\pi_{+} - \{ 2 \}$ is adjacent to any vertex  in  
$\pi_{-} - \{ 2\}$ and $2$ is a complete vertex of $\Delta_{\pi_{+} \cup \pi_{-}}$. 
\end{enumeratei}
\end{lemma}

\begin{lemma}[\mbox{\cite[Theorem~4.1]{W}}]
  \label{Sz}
  Let $S \cong {\rm Sz}(q)$, where $q = 2^a \geq 8$ and $a$ is odd. Then  $\Delta(S)$  has vertex set $\pi(S) = \{2 \} \cup \pi_0$,
  where  $\pi_0 = \pi((q-1)(q^2+1))$,  the set $\pi_0$ induces a complete subgraph of $\Delta(G)$ and  the vertex $2$
  is adjacent only to  the primes in $\pi_1 = \pi(q-1)$ in $\Delta(S)$. 
\end{lemma}

\begin{theorem}[\cite{W}]
    \label{CDsimple}
    Let $G$ be  a non-abelian simple group.
    \begin{enumeratei}
    \item If  $\Delta(G)$ is not a  complete graph, then
    $G$ is isomorphic to one of the following groups:
    ${\rm M}_{11}$, ${\rm M}_{23}$, ${\rm J}_1$, ${\rm A}_5 \cong {\rm PSL}_2(4) \cong {\rm PSL}_2(5)$,
    ${\rm A}_6 \cong {\rm PSL}_2(9)$, ${\rm A}_8 \cong {\rm PSL}_4(2)$, ${\rm PSL}_2(t^a)$ ($t$ prime, $t^a\geq 4$), ${\rm PSL}_3(t^a)$ for suitable values of \(t^a\) ($t$ prime), ${\rm PSU}_3(t^{2a})$ for suitable values of \(t^{a}\) ($t$ prime, $t^a> 2$),  ${\rm Sz}(2^a)$ ($a$ odd, $a \geq 3$). 
   \item Either  $\Delta(G)$ is  $2$-connected or  $G$ is isomorphic to one of the following groups:
     ${\rm M}_{11}$, ${\rm J}_1$, ${\rm PSL}_2(t^a)$ ($t$ prime, $t^a\geq 4$), ${\rm PSL}_3(4)$, ${\rm Sz}(2^a)$ ($a$ odd, $a \geq 3$ and $2^a-1$  a  prime number).
     \item $\Delta(G)$ is disconnected if and only if $G \cong {\rm PSL}_2(t^a)$ ($t$ prime, $t^a\geq 4$).
   \end{enumeratei}
  \end{theorem}
  \begin{proof}
    The list (a)  of the simple groups with a  non-complete degree graph is given in~\cite{W}, as well as the structure of the relevant graphs. In particular, one can immediately check that $\Delta({\rm{A}}_8)$ and $\Delta({\rm M}_{23})$ are $2$-connected.  
If $G =  {\rm PSL}_3(q)$, where $q= t^a \neq 4$, then  $\Delta(G)$  has vertex set $\{t \} \cup \pi_0$, where
    $\pi_0 = \pi((q-1)(q+1)(q^2+q+1))$; the set $\pi_0$ induces a complete subgraph of $\Delta(G)$, while the vertex $t$ is adjacent only to
    the primes in $\pi_1 = \pi((q+1)(q^2+q+1))$; so, all the vertices in $\pi_1$ are complete vertices of $\Delta(G)$ and,
    since $|\pi_1| \geq 2$, $\Delta(G)$ is $2$-connected.
    If $G =  {\rm PSU}_3(q^2)$, where $q= t^a >2$, then  $\Delta(G)$  has vertex set $\{t \} \cup \pi_0$, where
    $\pi_0 = \pi((q-1)(q+1)(q^2-q+1))$; the set $\pi_0$ induces a complete subgraph of $\Delta(G)$, while the vertex $t$ is adjacent only to
    the primes in $\pi_1 = \pi((q+1)(q^2-q+1))$; again, all the vertices in $\pi_1$ are complete vertices of $\Delta(G)$ and,
    since $|\pi_1| \geq 2$, $\Delta(G)$ is $2$-connected.
    Finally, if $G =  {\rm Sz}(q)$, where $q= 2^a \geq 8$ and $a$ is odd, then  by Lemma~\ref{Sz}  all the vertices in
    $\pi_1= \pi(q-1)$ are complete vertices of $\Delta(G)$.
    Hence, $\Delta(G)$ is not  $2$-connected if and only if $2^a-1$ is a prime power: as it is well known (see for instance \cite[Proposition 3.1]{MW}), this implies that both  $2^a-1$ 
    and $a$ are  prime numbers. This proves (b). Finally, (c) is Theorem 2.1 of~\cite{LW}. 
  \end{proof}

  
We recall that the character degree graphs of ${\rm M}_{11}$, ${\rm J}_1$ and ${\rm PSL}_3(4)$ are as follows (\cite{atlas}). 
  
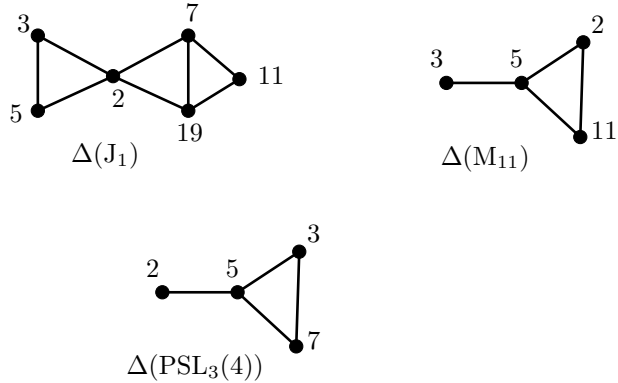
\begin{figure}[h]\label{JMP}
\caption{Degree graphs of ${\rm J}_1$, ${\rm M_{11}}$, ${\rm{PSL}}_3(4)$.}
\definecolor{ududff}{rgb}{0.30196078431372547,0.30196078431372547,1.}
\begin{minipage}{0.37\textwidth}
\vspace*{-7mm}
\begin{tikzpicture}[line cap=round,line join=round,>=triangle 45,x=1.0cm,y=1.0cm]
\clip(-0.08,1.54) rectangle (4.46,5.92);
\draw [line width=1.pt] (1.,4.)-- (1.,3.);
\draw [line width=1.pt] (2.,3.46)-- (1.,3.);
\draw [line width=1.pt] (1.,4.)-- (2.,3.46);
\draw [line width=1.pt] (2.,3.46)-- (3.,4.);
\draw [line width=1.pt] (3.,4.)-- (3.68,3.42);
\draw [line width=1.pt] (3.,3.)-- (3.68,3.42);
\draw [line width=1.pt] (2.,3.46)-- (3.,3.);
\draw [line width=1.pt] (3.,4.)-- (3.,3.);
\draw (0.6,4.4) node[anchor=north west] {$3$};
\draw (0.5,3.2) node[anchor=north west] {$5$};
\draw (1.86,3.4) node[anchor=north west] {$2$};
\draw (2.84,4.55) node[anchor=north west] {$7$};
\draw (2.72,2.95) node[anchor=north west] {$19$};

\draw (1.32,2.70) node[anchor=north west] {$\Delta({\rm{J}}_1)$};

\draw (3.8,3.7) node[anchor=north west] {$11$};
\begin{scriptsize}
\draw [fill=black] (1.,4.) circle (2.5pt);
\draw [fill=black] (1.,3.) circle (2.5pt);
\draw [fill=black] (2.,3.46) circle (2.5pt);
\draw [fill=black] (3.,4.) circle (2.5pt);
\draw [fill=black] (3.,3.) circle (2.5pt);
\draw [fill=black] (3.68,3.42) circle (2.5pt);
\end{scriptsize}
\end{tikzpicture}
\end{minipage}  \begin{minipage}{0.26\textwidth}
\begin{tikzpicture}[line cap=round,line join=round,>=triangle 45,x=1.0cm,y=1.0cm]
\clip(0.46,3.82) rectangle (3.98,6.20);
\draw [line width=1.pt] (1.,5.)-- (2.,5.);
\draw [line width=1.pt] (2.,5.)-- (2.82,5.54);
\draw [line width=1.pt] (2.82,5.54)-- (2.78,4.28);
\draw [line width=1.pt] (2.,5.)-- (2.78,4.28);
\draw (0.66,5.56) node[anchor=north west] {$3$};
\draw (1.72,5.56) node[anchor=north west] {$5$};
\draw (2.8,6.02) node[anchor=north west] {$2$};
\draw (2.8,4.6) node[anchor=north west] {$11$};

\draw (0.80,4.30) node[anchor=north west] {$\Delta({\rm{M}}_{11})$};

\begin{scriptsize}
\draw [fill=black](1.,5.) circle (2.5pt);
\draw[fill=black] (2.,5.) circle (2.5pt);
\draw [fill=black] (2.82,5.54) circle (2.5pt);
\draw [fill=black](2.78,4.28) circle (2.5pt);
\end{scriptsize}
\end{tikzpicture}
\end{minipage}   \begin{minipage}{0.36\textwidth}
\begin{tikzpicture}[line cap=round,line join=round,>=triangle 45,x=1.0cm,y=1.0cm]
\clip(0.46,3.82) rectangle (3.98,5.92);
\draw [line width=1.pt] (1.,5.)-- (2.,5.);
\draw [line width=1.pt] (2.,5.)-- (2.82,5.54);
\draw [line width=1.pt] (2.82,5.54)-- (2.78,4.28);
\draw [line width=1.pt] (2.,5.)-- (2.78,4.28);
\draw (0.66,5.56) node[anchor=north west] {$2$};
\draw (1.72,5.56) node[anchor=north west] {$5$};
\draw (2.8,6.02) node[anchor=north west] {$3$};
\draw (2.8,4.6) node[anchor=north west] {$7$};
\draw (0.4,4.3) node[anchor=north west] {$\Delta({\rm{PSL}}_3(4))$};

\begin{scriptsize}
\draw[fill=black](1.,5.) circle (2.5pt);
\draw  [fill=black](2.,5.) circle (2.5pt);
\draw [fill=black] (2.82,5.54) circle (2.5pt);
\draw  [fill=black] (2.78,4.28) circle (2.5pt);
\end{scriptsize}
\end{tikzpicture}
\end{minipage}
\end{figure}
  

\medskip

Recall that, for $a$ and $n$ integers larger than $1$, a prime divisor $q$ of $a^n-1$ is called a \emph{primitive prime divisor} if $q$ does not divide $a^b -1$ for all $1 \leq b <n$. In this case, $n$ is the order of $a$ modulo $q$, so $n$ divides $q-1$. 
It is known (\cite[Theorem~6.2]{MW}) that $a^n - 1$ always has primitive prime divisors except when $n = 2$ and $a= 2^c -1$ for some integer $c$, or when $n=6$ and $a= 2$.

We will make use of the following  well-known facts about linear groups.

\begin{lemma}\label{singer}
  Let $t$ be a prime number and $a  \geq 2$ an integer. Let $p$ be a primitive prime divisor of $t^a -1$ and let $X$ 
  be a subgroup of $G = {\rm GL}_a(t)$ such that $|X| = p$. Then $\cent GX$ is a cyclic group of order $t^a -1$.  
\end{lemma}
\begin{proof}
  Let $V$ be the natural module for $G$. 
  We observe that, as $p$ does not divide $t^b -1$ for every positive integer $b <a$, the subgroup  $X$ acts irreducibly on $V$.
  Hence,  by Schur's lemma, the ring $\mathbb{K} = {\rm{End}}_{\mathbb{F}_t[X]}(V)$ of the $\mathbb{F}_t[X]$-endomorphisms of $V$  
  is a finite field and, since $\cent GX$ is contained in the multiplicative group of $\mathbb{K}$, it follows that $\cent GX$
  is cyclic.
  By~\cite[II.7.3]{Hu} $G$  contains a cyclic subgroup $C$ of order $t^{a} - 1$ (a Singer cycle) and  $\cent GC = C$; so,
  $C$ is a maximal cyclic subgroup of $G$. 
  Observing that the $p$-part of $|G|$ coincides with the $p$-part of $|C|$, we can assume that $X$ is a subgroup of $C$.
  Hence $C \subseteq \cent GX$ and, since $\cent GX$ is cyclic, we conclude that $\cent GX = C$.
\end{proof}

\begin{lemma}\label{clg}
  For a positive integer $a$ and a prime $t$, let $G = {\rm GL}_{2a}(t)$ and let $K \leq L \leq G$ be such that
  $L \cong {\rm GL}_{2}(t^a)$ and   $K \cong {\rm SL}_{2}(t^a)$. Then $\cent GK = \zent L$ is a cyclic group of order $t^a -1$.  
\end{lemma}
\begin{proof}
  Let  $C$ be a Singer cycle of $L$, so, $|C| = t^{2a} - 1$  and  $C$ is a maximal cyclic subgroup of $G$.
  Assume first that there exists a primitive prime divisor $p$ of $t^{2a}-1$. Hence,  $p$ divides $|K|$ and we consider a subgroup
  $X \leq K \cap C$ with $|X| = p$.
As $C \leq \cent GX$, by Lemma~\ref{singer} we conclude that $\cent GX = C$.
  Thus, $\cent GK \leq \cent GX \leq L$ and it is well known that $\cent GK = \cent LK = \zent L$ is cyclic of order $t^a -1$.

  If there is no primitive prime divisor of $t^{2a}-1$, then either $a=1$ and  we are done as $L = G$, or
  $t^{2a} = 2^6$ and we can argue as above with a cyclic subgroup $X$ of $K$ such that  $|X| = 9$.
\end{proof}

As another important reference, we recall the characterization of the non-solvable groups whose character degree graph is disconnected (\cite{LW}).

\begin{theorem}[\mbox{\cite[Theorem 4.1 and Theorem~6.3]{LW}}]
\label{LewisWhite}
Let \(G\) be a non-solvable group. Then:
\begin{description}
\item[(I)] \(\Delta(G)\) has two connected components if and only if there exist normal subgroups \(N\subseteq K\) of $G$ such that, setting \(C/N=\cent{G/N}{K/N}\), the following conditions hold.
  \begin{enumeratei}
  \item \(K/N\cong\PSL{t^a}\), where \(t\) is a prime with \(t^a\geq 4\).
  \item \(G/K\) is abelian.
  \item If \(t^a\neq 4\), then \(t\) does not divide \(|G/CK|\).
  \item If \(N\neq 1\), then either \(K\cong\SL{t^a}\) or there exists a minimal normal subgroup \(L\) of \(G\) such that \(K/L\cong\SL{t^a}\) and \(L\) is isomorphic to the natural module for \(K/L\).
  \item If \(t=2\) or \(t^a=5\), then either \(CK\neq G\) or \(N\neq 1\).
  \item If \(t=2\) and \(K\) is an in {\rm{(d)}} in the case \(K\not\cong \SL{t^a}\), then every non-principal character in \(\irr L\) extends to its inertia subgroup in \(G\).
\end{enumeratei}
\item[(II)] \(\Delta(G)\) has three connected components if and only if $G \cong \SL{2^a} \times A$, where $a \geq 2$ and
  $A$ is an abelian group. 
\end{description}
\end{theorem}

\begin{rem}\label{verticesdisc}
Let \(G\) be a non-solvable group whose character degree graph \(\Delta(G)\) has two connected components. Hence the structure of \(G\) is described by the conditions (a)--(f) in the above statement; observe also that (adopting the notation of Theorem~\ref{LewisWhite}) \(K=G'\) coincides with the last term in the derived series of \(G\), whereas \(C\) coincides with the solvable radical \(R\) of \(G\) and \(N=K\cap R\). Then, we claim that \emph{the vertex set of \(\Delta(G)\) consists of the primes in \(\pi(G/R)\)}. (Note that this claim is obviously true in the case when \(\Delta(G)\) has three connected components, so it holds in fact whenever \(\Delta(G)\) is disconnected.)

In fact, since \(G/R\) is an almost-simple group, every prime in \(\pi(G/R)\) is a vertex of \(\Delta(G/R)\), hence of \(\Delta(G)\). 
On the other hand, let \(u\) be a prime that does not lie in \(\pi(G/R)\), and let \(U\) be a Sylow \(u\)-subgroup of \(G\) (which lies in \(R\)). Assume that \(G\) is as in (d) of Theorem~\ref{LewisWhite} (so \(N\neq 1\)), in the case when there exists a minimal normal subgroup \(L\) of \(G\) such that \(K/L\cong\SL{t^a}\) and \(L\) is isomorphic to the natural module for \(K/L\). Since \(R/N\cong KR/K\subseteq G/K\) is abelian, \(UN/N\) is abelian and normal in \(R/N\), hence it is an abelian normal Sylow \(u\)-subgroup of \(G/N\) and, by Proposition~\ref{Ito}, \(u\not\in\V{G/N}\) (note that, by this argument, our claim is proved if \(N\) is assumed to be trivial instead). As \(N/L\) is a central subgroup of order at most \(2\) of \(G/L\), it is immediate to see that \(UL/L\) is abelian and normal in \(R/L\) (hence we are done if we are in case (d) with \(K\cong\SL{t^a}\)). Finally, since \([K/N,R/N]=1\), clearly we have \(\cent{K/N}U=K/N\); but in fact, by coprimality, we have \(K/N\cong(K/L)/(N/L)=\cent{K/L}{U}/(N/L)\), whence \(K/L=\cent{K/L}{U}\). In particular, setting \(\o{G}=G/\cent G L\) and adopting the bar convention, we get \([\o U, \o K]=1\) because obviously \(L\subseteq\cent G L\). Now \(\o G\) embeds in \({\rm{GL}}_{2a}(t)\), and \(\o K\) is a subgroup of \(\o G\) isomorphic to \(\SL{t^a}\): hence we are in a position to apply Lemma~\ref{clg} obtaining that \(\o U\) must be trivial, because \(t^a-1\) is a divisor of $|G/R|$ and is therefore coprime with \(u\).  It easily follows that \(U\) is an abelian normal Sylow \(u\)-subgroup of \(G\), so \(u\not\in\V G\), and our claim is proved.
\end{rem}





The following result will come into play later on.

\begin{theorem}\label{MoretoTiep}
  Let $G$ be an almost-simple group with socle $S$, and let $\delta = \pi(G) - \pi(S)$.
  If $\delta \neq \emptyset$, then $S$ is a simple group of Lie type, and every vertex in \(\delta\) is adjacent to every other vertex of $\Delta(G)$ that is not the characteristic of $S$.
\end{theorem}
\begin{proof}
  Assume $\delta \neq \emptyset$ and let $p \in \delta$. Since $p$ divides $|{\rm{Aut}}(S)|$ but $p$ does not divide $|S|$,
  by the classification of the finite simple groups we deduce that $S$ is a simple group of Lie type in characteristic, say, $t$.
  By Theorem~2.7 of \cite{MT}, $p$ is adjacent in $\Delta(G)$ to every vertex $q\in \pi(S)$ such that $q\neq t$.
  Moreover, $G/S$ has a normal (in fact,  central) and cyclic Hall $\delta$-subgroup $D/S$ (\cite[Lemma 2.10]{MT}); so,
by Schur-Zassenhaus theorem $G$ has a cyclic Hall $\delta$-subgroup $A$ and by Lemma 2.6 of \cite{MT} $A$ has a regular orbit on $\irr S$. Hence, by Clifford correspondence, there is a $\chi \in \irr G$ such that $|A|$ divides $\chi(1)$ and we conclude that,
in particular, the subgraph of $\Delta(G)$ induced by $\delta$ is a clique. 
\end{proof}

To avoid any ambiguity in the previous statement, we mention that no simple group of Lie type in double characteristic has an outer automorphism whose order is coprime to the order of the group itself.



\smallskip
In the course of our analysis, we will also  make use of special types of actions of groups on modules. 
Let \(H\) and \(V\) be finite groups, and assume that \(H\) acts by automorphisms on \(V\). Given a prime number  \(q\), we say that \emph{the pair \((H,V)\) satisfies the condition \(\nq\) } if  $q$ divides $|H: \cent HV|$ and, for every non-trivial \(v\in V\), there exists a Sylow \(q\)-subgroup \(Q\) of \(H\) such that \(Q\trianglelefteq \cent H v\) (see \cite{C}).

If $(H, V)$ satisfies $\nq$ then \(V\) turns out to be an elementary abelian \(r\)-group for a suitable prime \(r\), and \(V\) is in fact an \emph{irreducible} module for \(H\) over the field $\mathbb{F}_r$ with \(r\) elements  (see~Lemma~4 of~\cite{Z}). 


\begin{lemma}\label{SL2Nq}
  Let $t$, $q$, $r$  be prime numbers, let $H = {\rm SL}_2(t^a)$ (with $t^a \geq 4$)  and  let  $V$ be an $\mathbb{F}_r[H]$-module.
  Then $(H, V)$ satisfies $\nq $ if and only if 
 either $t^a = 5$ and $V$ is the natural module for $H/\cent HV \cong {\rm SL}_2(4)$ or $V$ is faithful and 
  one of the following holds.
  \begin{description}
  \item[(1)] $t = q = r$ and $V$ is the natural $\mathbb{F}_r[H]$-module (so $|V| = t^{2a}$); 
  \item[(2)] $q = r = 3$ and $(t^a, |V|) \in \{(5, 3^4), (13, 3^6)\}$. 
  \end{description}
\end{lemma}
\begin{proof}
  First, we consider the ``only if'' part of the statement. 
  If $t \neq 2 \neq q$, this is Proposition~13 of~\cite{C}, while for $t = 2$ and $q$ odd it follows from Lemma 3.3 of~\cite{ACDPS}.
  We can therefore assume $q = 2$, thus also $r = 2$ by part (2) of Proposition~8 of \cite{C}. 
  Next, we  suppose that $t$ is odd, and we  prove that $t^a = 5$ and that $V$ is isomorphic to the natural module for
  $H/\cent HV \cong {\rm SL}_2(4)$.
  Observe that, by the condition $\mathcal{N}_2$, the kernel $\cent HV$ of the action is the cyclic subgroup $\zent H$ of order $2$. 
Let $T$ be a Sylow $t$-subgroup of $H$. If there exists a non-trivial element $x \in T$ such that $\cent Vx \neq 1$, 
then $x$ normalizes a Sylow $2$-subgroup of $H$ by the condition $\mathcal{N}_2$: in view of the structure of the subgroups of ${\rm PSL}_2(t^a)$ (\cite[Hauptsatz II.8.27]{Hu}), this can happen only if  $t= 3$.

Let us first assume $t \neq 3$. Hence $T$ acts fixed-point freely on $V$ and, as $T$ is elementary abelian, it follows that  $a= 1$.
Let $Q$ be a Sylow $2$-subgroup of $H$.
  It is well known that $Q$ is a  generalized quaternion group (since $H$ has a unique involution) and it easily follows
  that $\norm HQ = Q$ if $t^2 \equiv 1 \pmod{16}$ and, recalling 
  that
  $H/\zent H$ has subgroups isomorphic to $A_4$ as $t \neq 2$ (\cite[Hauptsatz II.8.27]{Hu}), that $\norm HQ \cong {\rm SL}_2(3)$ if $t^2 \not\equiv 1 \pmod{16}$. 
  Now, if  $p$ is an odd divisor of $t-1$, then  $H$ has a Frobenius subgroup $TP$, with kernel $T$ and $P$ cyclic of order $p$.
  By Theorem 15.16 of~\cite{Is} $\cent VP$ is non-trivial and hence, again by the condition $\mathcal{N}_2$, $P$ normalizes
  a Sylow $2$-subgroup of $H$.
  It follows that either $t-1 = 2^k \cdot 3$, with $k \leq 2$,  or $t-1 = 2^k$, $k \geq 2$.   
  The condition $\mathcal{N}_2$ implies that the set $V^{\#}$ of the non-trivial elements of $V$ is a disjoint union of the
  sets $\cent VQ - \{1 \}$, where $Q$ varies in the set of the Sylow $2$-subgroups of $H$.
  Hence, writing $|V| = 2^d$,  $|\cent VQ|= 2^b$ and denoting by $n_2(H)$ the number of Sylow $2$-subgroups of $H$, we have 
  \begin{equation}
    \label{eq:c}
  n_2(H) = \frac{2^d-1}{2^b -1} = 1 + 2^b + 2^{2b} + \cdots + 2^{(\frac{d}{b} -1)b}.   
  \end{equation}
   If $t = 7$, then $n_2(H) = 21$ and hence $d= 6$ (and $b= 2$), which is impossible as $H$ has no irreducible module of dimension $6$
  over $\mathbb{F}_2$.
  If $t = 13$, then $n_2(H) = 91$, which has no binary expansion of the form (\ref{eq:c}).
  Hence, $t = 2^k +1$, with $k \geq 2$, and $n_2(H) = t(t+1)/2 = 1 + 2^{k-1} + 2^k + 2^{2k-1}$.
  By~(\ref{eq:c}), it follows that $k= 2$ and $t = 5$, so $d = 4$ (and $b = 1$).
  Of  the two  $H$-module of dimension $4$ over $\mathbb{F}_2$, the only one
  that satisfies $\mathcal{N}_2$ is the natural module for $H/\cent HV \cong {\rm SL}_2(4)$. 

  On the other hand, assume $t = 3$; so $a \geq 2$ and a Sylow $3$-subgroup of $H$, being non-cyclic, cannot act fixed-point freely on $V$.
  Since the normalizer of a Sylow $3$-subgroup of $H$ acts transitively on its subgroups of order $3$,   all the subgroups of order $3$ of $H$ are conjugate in $H$.  It follows that each of them has a non-trivial
  centralizer in $V$ and hence, by the condition $\mathcal{N}_2$ and the discussion above, we deduce that
  $a$ is odd and that $\cent Hv \cong \rm{SL}_2(3)$ for every non-trivial $v \in V$.
  In particular, all the $H$-orbits of non-trivial vectors of $V$ have the same size (i.e. $H$ acts $\frac{1}{2}$-transitively on $V$).
  Hence, recalling that $H$ is non-solvable and that the orbit sizes are odd by the condition $\mathcal{N}_2$,
  Theorem~6 of \cite{GLPST} yields that $H$ acts transitively on the non-trivial elements of $V$.
  So, using the classification of the doubly transitive affine permutation groups (see Appendix~1 of~\cite{Li}) and
  recalling that $t \neq q$ (as $t= 3$ and $q= 2$), we deduce that the only possiblity is $H \cong \rm{SL}_2(9)$ and
  $|V| = 2^4$; but in this case, for any  non-trivial $v \in V$,  $\cent Hv \cong S_4$, a contradiction. 
  
  So, we are left with the case $t = q = r =  2$.
In this situation, equation  (\ref{eq:c}) gives
  $2^d - 1 = (2^a +1)(2^b -1) = 2^{a+b} +2^b - 2^a - 1$ and hence $a = b$ and $d = 2a$. 
By Lemma 3.12 of~\cite{PR}, we conclude again that $V$ is the natural $\mathbb{F}_2[H]$-module. 

\medskip
Conversely, it is easily checked that each isomorphism type of the modules described in the statement (there are two conjugacy classes both of subgroups isomorphic
to ${\rm SL}_2(5)$  in ${\rm GL}_4(3)$ and of subgroups isomorphic to ${\rm SL}_2(13)$  in ${\rm GL}_6(3)$)  satisfies the condition $\nq$. 
\end{proof}

\begin{rem}\label{remPSL}
  We observe that if $H = {\rm PSL}_2(t^a)$ and $V$ is an $\mathbb{F}_r [H]$-module such that $(H, V)$ satisfies the condition  $\mathcal{N}_q$,
  then, seeing  $V$ in a natural way  as a
  module for $K = {\rm SL}_2(t^a)$ (with $\zent K$ in the kernel of the action), $(K, V)$ satisfies $\mathcal{N}_q$ too.
  In fact, $|K:\cent Kv| = |H:\cent Hv|$ for every $v \in V$ and $\cent Kv$ contains a Sylow $q$-subgroup of $K$ as a normal subgroup, for every non-trivial $v \in V$.
We hence conclude that  Lemma~\ref{SL2Nq} describes also all possible module actions of ${\rm PSL}_2(t^a)$ that satisfy the condition $\mathcal{N}_q$, for any prime $q$.  
\end{rem}

\begin{lemma}\label{Nqvari}
  Let $K$ be a quasi-simple group and let $V$ be a finite and faithful $K$-module such that, for a prime $q$,
  $(K, V)$ satisfies the condition $\mathcal{N}_q$. Then $K/\zent K \not\cong J_1, {\rm{M}}_{11}, {\rm PSL}_3(4)$.
  Moreover, if $K/\zent K \cong {\rm Sz}(2^a)$, then $q \neq 2$. 
\end{lemma}
\begin{proof}
  Write $|V| = r^d$, where $r$ is a prime, and for a Sylow $q$-subgroup $Q$ of $K$ let $|\cent VQ| = r^b$.
  Since $(K, V)$ satisfies the condition $\mathcal{N}_q$, Proposition 2.3(b,c) of~\cite{DKP} yields that $b$ is a proper divisor of  $d$ and, setting $c = d/b$, 
  $(r^d - 1)/(r^b -1) = 1 + r^b + r^{2b} + \cdots + r^{(c -1)b}$ coincides with the  number $n_q(K)$ of the Sylow $q$-subgroups of $K$.
  Moreover, Theorem 2.5 of~\cite{DKP} implies that  $r$ divides $|K|$.
  A check of the $r$-adic representations of $n_q(K)$, for every choice of prime divisors  $r, q$ of $|K|$, excludes the cases    $K/\zent K \cong J_1, M_{11}, {\rm PSL}_3(4)$.

  Finally, working by contradiction we assume  that  $K/\zent K \cong {\rm Sz}(2^a)$ for an odd integer $a \geq 3$  and
  that $(K, V)$ satisfies the condition $\mathcal{N}_2$. 
  Then $r=2$ by~\cite[Proposition 2.3(a)]{DKP} and, as $n_2(K) = 2^{2a}+1$, 
  we deduce that $b = 2a$ and that $d = 4a$.
  The Schur multiplier of ${\rm Sz}(2^a)$ is a $2$-group for $a = 3$,  and it is trivial for $a \geq 5$ (\cite[Theorem 1 and Theorem 2]{AG}).
  Hence, as the action of $K$ on $V$ is faithful and satisfies the condition  $\mathcal{N}_2$, we conclude that in any case $\zent K = 1$. 
  By  the structure of the absolutely  irreducible modules for the Suzuki simple groups on fields of characteristic $2$
  (\cite[Lemma 1]{M}), we deduce that $V$ is (any Galois conjugate of) the natural $\mathbb{F}_{2^a}[K]$-module seen as $\mathbb{F}_2[K]$-module, and hence $|\cent VS|= 2^a$ and $b= a$, a contradiction. 
\end{proof}

Next, an improvement of Proposition~3.2 of \cite{ACDPS}.

\begin{proposition}\label{P3.12+}
  Let $H$ be a non-solvable group and let $V$ be a finite and faithful $H$-module.
  If $(H, V)$ satisfies the condition $\nq$ for some prime $q$, then the solvable radical of $H$ is cyclic. 
\end{proposition}
\begin{proof}
  By Lemma~4 of~\cite{Z} we know that $V$ is an irreducible module; write $|V| = r^a$, where $r$ is a prime and $a \geq 2$. 
  By Proposition~3.2 of \cite{ACDPS}, we can assume that $r = q$. Let $R$ be the solvable radical of $H$ and let $Q$ be a
  Sylow $q$-subgroup of $H$.

  We consider first the case $a = 2$. Then we can see $H$ as a subgroup of ${\rm GL}_2(q)$ and $H_0 = H \cap {\rm SL}_2(q)$ is
  non-solvable; moreover, $(H_0, V)$ satisfies $\nq$. If $H_0$ is a proper (non-solvable) subgroup of $ {\rm SL}_2(q)$, then the classification of the subgroups of ${\rm SL}_2(q)$ yields  $H_0 \cong A_5 \cong {\rm SL}_2(4)$:  Lemma~\ref{SL2Nq} yields now that   $|V| = 2^4$, not our case because we are assuming $|V|=r^2$.
  Hence, ${\rm SL}_2(q)=H_0$ and therefore $R \leq \cent{{\rm GL}_2(q)}{H_0} \leq \zent{{\rm GL}_2(q)}$ is cyclic.

  Next, we assume that $|V| = 2^6$. So, by Lemma 3.1 of ~\cite{ACDPS} $H$ has a normal subgroup isomorphic to ${\rm SL}_2(8)$.
  Since $\rm{GL}_6(2)$ has a unique conjugacy class of maximal subgroups having a section isomorphic to ${\rm SL}_2(8)$,
  and they are isomorphic to ${\rm GL}_2(8) : C_3 = (C_7 \times {\rm SL}_2(8)):C_3$, we conclude that  also in this case $R$
  is cyclic.

  We can hence assume that $|V| = q^a \neq q^2, 2^6$; so,  there exists a primitive prime divisor $t$ of $q^a -1$  (recall also that 
 $a$ divides $t-1$).
  Let $T$ be a Sylow $t$-subgroup of $H$. Since $(H, V)$ satisfies the condition  $\nq$, then $(|V|-1)/(|\cent VQ|-1) = |\syl qH|$ 
  divides $|H|$, thus $T \neq 1$.
  We also observe that $t$ does not divide $|R|$. In fact, if $T_0 = T \cap R \neq 1$, then $\zent{T_0}$,  a normal
  abelian subgroup of $\norm H{T_0}$,  acts irreducibly on $V$ and hence  by \cite[Theorem~2.1]{MW} $\norm H{T_0}$ is solvable, which is a contradiction as $H = R \norm H{T_0}$ by the Frattini argument.

  For every prime divisor $p$ of the order of $R$, there exists a $T$-invariant Sylow $p$-subgroup $P$ of $R$
  and by  Lemma~6 of \cite{C} we get $[P, T] = 1$ if $p \neq 2$. It follows that $R = \cent RT D$, where $D$ is a $T$-invariant Sylow $2$-subgroup of $R$.
  Write $E = [R, T]$; then $E = [\cent RT D, T] = [D,T] \leq D$ is a normal $2$-subgroup of $R$. So, $E \leq \fit R \leq \fit H$.
  We observe that $\cent HT \leq \cent{{\rm GL}_a(q)}T$ is cyclic by~Lemma~\ref{singer}.

  It is enough to prove that $E$ is cyclic: in this case, in fact, $T$ centralizes $E$ (as ${\rm Aut}(E)$ is
  a $2$-group and $t \neq 2$) and then by coprimality $[R, T] = [R, T, T] = [E, T] =1$, so $R \leq \cent HT$ is cyclic.

   Working by contradiction, we assume that $E$ is non-cyclic. 
   If $A$ is a characteristic abelian subgroup of $E$, then $V$ is an irreducible $AT$-module (because $t$ is a primitive prime divisor of \(q^a-1\)); moreover, the restriction of $V$ to $A$ is homogeneous, since otherwise it would have at least $t$ homogeneous components, yielding the contradiction $a\geq t>a$; thus, $A$ is cyclic. As a consequence, $T$ centralizes every characteristic abelian subgroup of $E$. Since $[E,T]=E$,  \cite[24.7]{A} yields that $E$ is an extraspecial $2$-group and $\zent E = \cent ET$.
   Write $|E| = 2^{2n +1}$.
   By applying~\cite[Corollary 2.6]{MW} to an irreducible constituent of the homogeneous module $V_E$, we see that $2^n$ divides the dimension $a$ of $V$ over $\mathbb{F}_q$ and then $2^n +1 \leq a +1 \leq t$.
   On the other hand, $|T|$ divides $2^n \pm 1$ by ~\cite[Satz V.17.13]{Hu} and we deduce that $t = 2^n+1 = a+1$.
   Also, by~\cite[Corollary 2.6]{MW} $V$ is an absolutely irreducible $E$-module.
   As $a > 2$, then $n \geq 2$ and there exists a non-central involution $y$ in $E$.
   Let $Y  = \cent Hy$ and $U = \cent Vy$. We remark  that $q \neq 2$, as $E$ acts 
faithfully and irreducibly on $V$. 
   Hence, it follows that  $|U| = |V|^{1/2}$. In fact, if 
$\phi$ is the Brauer character corresponding to $V$, then $\phi \in \irr E$,  $\phi(1) = 2^n$ and $\phi(y) = 0$. 
So,  $\dim_{\mathbb{F}_q}(U) = [\phi_{\langle y \rangle}, 1_{\langle y \rangle}] = \phi(1)/2 = a/2$. 

Now, if $\cent HU$ contains a Sylow $q$-subgroup of $H$, that (up to conjugation) we can assume to be $Q$,  then
$U = \cent VQ$. In fact, $U \leq  \cent VQ$ and $|\cent VQ| \leq |V|^{1/2}$, because  the condition $\nq$ implies that the centralizers
in $V$ of two distinct Sylow $q$-subgroups of $H$ have trivial intersection. 
So, $\cent V{\langle y \rangle} = \cent VQ$ and, in particular,  $y$ normalizes $Q$ because of the condition $\nq$.
As $y  \in \oh 2R \nor H$, it follows that $[y, Q] \leq \oh 2R \cap Q = 1$ and hence  
by  Thompson's $A\times B$ Lemma (\cite[24.2]{A}) we deduce that $y$ acts trivially
on $V$, a contradiction. 

Hence, $q$ divides $|H: \cent HU|$ and, since $q \neq 2$ and $\langle y \rangle \leq \oh 2R \nor H$,  \cite[Lemma~2.4]{DKP} yields that $(Y,U)$ satisfies the condition $\nq$, as well.
Set $\o Y = Y/\cent YU$. If $\o Y$ is solvable then, recalling that \(a\neq 2\), Proposition~9 in \cite{C} ensures that \(G\) is a group of semilinear maps on \(V\); now, $\fit{\o Y}$ is cyclic by Lemma 6.4 and Corollary 6.6 of~\cite{MW}.
If $\o Y$ is non-solvable, then  we anyway get  that $\fit{\o Y}$ is cyclic, working by induction on $|V|$.  
Now, $\cent E{\langle y \rangle} = \langle y \rangle \times E_0$, 
where $E_0$ is an extraspecial group of order $2^{2(n-1)+1}$. 
We observe that $E_0 \leq E \cap Y \leq \fit H \cap Y \leq \fit Y$ and that  $\cent {E_0}U  = 1$, as $\cent {E_0}U$ is a normal subgroup of $E_0$ and it intersects trivially $ \zent{E_0} = \zent E =  \langle z \rangle$, because $z$ acts as the inversion on $V$.
So $E_0 \cong \o{E_0} \leq \fit{\o Y}$, which is cyclic, the final contradiction. 
\end{proof}

We  now derive a first set of consequences of the assumption that $\Delta(G)$ has connectivity degree $1$, in the case that the group $G$ has a non-abelian minimal normal subgroup (case that, of course, includes all almost-simple groups). 

\begin{proposition}\label{nonabeliansocle}
  Let \(G\) be a group which has a non-abelian minimal normal subgroup $M$. If  \(\Delta(G)\) is connected and
 has a   cut-vertex \(p\), then the following conditions are satisfied.
\begin{enumeratei}
\item $\V G=\pi(G/{\cent G M})\cup\V {\cent G M}$.
\item $\V{\cent G M}\sbs \{p\}$, and \(R(G)=\cent G M\). Further, if $\V{\cent G M}= \{p\}$, then $p$ is a complete vertex of \(\Delta(G)\).
\item $M$ is a simple group.
\item $\Delta(M)$ is either disconnected or it is connected with cut-vertex $p$.
\item \(M\) is isomorphic to one of the groups of the following list:
  $$\mathcal{L} = \{ {\rm{PSL}}_2(t^a)\;({\rm{ with }}\;t^a \geq 4),\; {\rm{Sz}}(2^a)\; ({\rm{ with }}\;a \geq 3,\; a\; {\rm{ odd}}),\; {\rm{PSL}}_3(4),\;
  {\rm{M}}_{11},\; {\rm{J}}_1\}.$$ 
\end{enumeratei}
\end{proposition}

\begin{proof} 
  Setting \(C=\cent G M\), we first prove that  \(\V G=\pi(G/C)\cup\V C\).
  Certainly $\pi(G/C)\cup\V C \sbs \V G$. If \(q\in \V G - \pi(G/C)\), then \(C\) contains a Sylow \(q\)-subgroup \(Q\) of \(G\). Now, if \(q\) does not belong to  \(\V C\), then \(Q\) is normal in \(C\) (thus, in \(G\)) and abelian, which yields the contradiction \(q\not\in\V G\). This proves (a).

Next, we show that for every choice of \(q \in \pi(G/C)\) and \(r \in \V C\), the vertices \(q\) and \(r\) are adjacent in \(\Delta(G)\). Observe in fact that, in view of \cite[Proposition~2.10(b)]{DKP}, for every \(q \in \pi(G/C)\) there exists \(\theta \in \irr M\) such that \(q\) divides the degree of every \(\chi\in\irr{G|\theta}\). Thus, if \(\phi\in\irr C\) is such that \(r\) divides \(\phi(1)\), then any irreducible character of \(G\) lying over \(\theta\times\phi\in\irr{MC}\) has a degree divisible by \(qr\), as claimed. 

Recalling that $\pi(G/C) = \V{G/C}$ contains at least three elements (as $G/C$ has a non-abelian composition factor),
we claim that the previous paragraph forces $\V{C} \subseteq \{p\}$.  To see this, assume that   \(\V C-\{p\}\) contains a prime  \(r\); 
let  $ q \in  \pi(G/C)-\{p\}$, and choose  two distinct primes \(s\) and \(u\) in \(\V G-\{p\}\). If one of them lies in \(\pi(G/C)\) and the other in \(\V C\), then we know that \(s\), \(u\) are adjacent in \(\Delta(G)\);  on the other hand, if both lie in \(\pi(G/C)\) or \(\V C\), then either \(s-r-u\) or \(s-q-u\) is a path in \(\Delta(G)\). This contradicts the fact that
$\Delta(G) - p$ is disconnected, so  $\V{C} \subseteq \{p\}$.
We also observe that, if $\V{C} = \{p\}$, then by the previous paragraph $p$ is a complete vertex of $\Delta(G)$. 
Finally, since  $\V{C} \subseteq \{p\}$, Proposition~\ref{Ito} yields that \(C\) has an abelian normal \(p\)-complement, and hence $C \subseteq R(G)$. As the other inclusion is clear,   we get \(C=R(G) \) and (b) is proved.

Let us now show that $M$ is simple. Assuming the contrary, Proposition 2.10(c) of \cite{DKP} yields that \(\pi(G/C)\) is a clique of \(\Delta(G)\), against the fact that \(\Delta(G)-p\) is not connected.  
This proves (c). 

In order to prove (d) and (e), 
we set $\o G = G/C$. Recalling that $\Delta(\o G)$ is a subgraph of $\Delta(G)$ and that $\V G = \pi(\o G) \cup \{ p\}$,
it is clear that if \(p\) does not divide \(|\o G|\), then \(\Delta(\o G)\) is disconnected.
Since $\o G$ is an almost-simple group with socle isomorphic to $M$, by Theorem~\ref{LewisWhite} it follows that
\(M\cong {\rm{PSL}}_2(t^a)\) for \(t^a\geq 4\) and hence $\Delta(M)$ is disconnected, giving us both (d) and (e).
So,  we may assume \(p\in\pi(\o G)\) and hence $\V G = \pi(\o G)$.  
If $\pi(M) = \pi(\o G)$, then $\Delta(M)$ is a subgraph of $\Delta(G)$ with the same vertex set, so $\Delta(M)$ is either
disconnected or it is connected with cut-vertex $p$, hence (d) and (e) follow by part (b) of  Theorem~\ref{CDsimple}. 
If \(\pi(\o G)-\pi(M)\) is non-empty, then Theorem~\ref{MoretoTiep} ensures that
\(M\) is a simple group of Lie type (in characteristic $t$, say), and that every prime in  \(\pi(\o G)-\pi(M)\) is adjacent
in $\Delta(G)$ to every other vertex different from $t$. In particular, $\Delta(G) - t$ is a connected graph and
hence $p \neq t$.
Now, if $\Delta(M)$ is disconnected, then \(M\cong {\rm{PSL}}_2(t^a)\) for \(t^a\geq 4\) by part (c) of  Theorem~\ref{CDsimple}.
If $\Delta(M)$ is connected, then necessarily $p \in \pi(M)$ and $p$ is the only vertex adjacent to $t$ in $\Delta(G)$.
It follows that $\Delta(M)$ is connected with cut-vertex $p$, and hence by  part (b) of  Theorem~\ref{CDsimple} 
$M$ is isomorphic to one of the simple groups in the list $\mathcal{L}$. 
The proof of (d) and (e) is now complete.
\end{proof}

Next, we state a result that is going to be of fundamental importance for the rest of our work.

\begin{theorem}
\label{0.2}
Let $G$ be a non-solvable group such that $\Delta(G)$ is connected and it has a  cut-vertex $p$. Then, setting $R=R(G)$, we have that $G/R$ is an almost-simple group such that $\V G=\pi(G/R)\cup\{p\}$. Furthermore, the socle $X/R$ of $G/R$ is such that $\Delta(X/R)$ is either disconnected or connected with cut-vertex $p$, and $X/R$ is isomorphic to one of the simple groups in the list $\mathcal{L}$
of Proposition~$\ref{nonabeliansocle}$.
\end{theorem}

\begin{proof} 
  We will prove that if   $G$ is a non-solvable group such that $\Delta(G)$ is either disconnected or  connected with cut-vertex $p$,  then  $G/R$ is an almost-simple group  with socle $X/R \in \mathcal{L}$ such that $\Delta(X/R)$ is either disconnected or
  connected with cut-vertex $p$, and  that either $\V G = \pi(G/R)$ or  
  $\V G=\pi(G/R)\cup\{p\}$ (respectively). 
  Assuming that the statement is false, let $G$ be a counterexample having minimal order.
We observe that, by Theorem~\ref{LewisWhite} and Remark~\ref{verticesdisc}, $\Delta(G)$ is connected  and we denote by $p$ a cut-vertex of $\Delta(G)$. 
By Proposition~\ref{nonabeliansocle}, we see that $G$ does not have any non-abelian minimal normal subgroup.
We remark that the choice of $G$ implies that, for every solvable normal subgroup $N$ of $G$, the set $\V G$ is strictly larger than $\V{G/N}\cup\{p\}$. In fact, assuming the contrary, as $\Delta(G/N)$ is a subgraph of $\Delta(G)$, we get that $\Delta(G/N)$ is either disconnected or connected with cut-vertex $p$. Now, by the minimality of $G$  we have that $G/R\cong (G/N)/(R/N)$ is an almost-simple group whose socle \(X/R\in\mathcal{L}\) is such that \(\Delta(X/R)\) is either disconnected or connected with cut-vertex \(p\), and $\V{G/N}\cup\{p\} \subseteq \V{G/R}\cup\{p\}$; so, $G$ would not be a counterexample. 

Moreover, we claim that  the generalized Fitting subgroup $\fitg{G}$ of $G$  and the Fitting subgroup $\fit G$ of $G$ coincide.
In fact, let $\lay G$ denote the subgroup generated by all the components of $G$. Arguing by contradiction, assume $\lay G\neq 1$: as $G$ does not have any non-abelian minimal normal subgroups, we have $Z=\zent{\lay G}\neq 1$. Now, as observed in the
previous paragraph, there exists a prime $q\in\V G-\V{G/Z}$; so, if $Q\in\syl q G$, then $QZ$ is a normal subgroup of $G$ and
$QZ/Z$ is abelian. Note that $q$ does not divide $|\lay G/Z|$, because $\pi(\lay G/Z)=\V{\lay G/Z}\subseteq\V{G/Z}$.
Since $QZ$ is solvable, it centralizes $\lay G$, and so $Z$ is central in $QZ$. Therefore $Q$ is normal in $G$, and now
we must have $Q\cap Z\neq 1$, as otherwise $Q\cong QZ/Z$ would be an abelian normal Sylow $q$-subgroup of $G$, against
$q \in \V G$. The conclusion is that $\lay G$ has a non-trivial central Sylow $q$-subgroup, a contradiction.

Next we claim that the Frattini subgroup $\frat G$ of $G$ is trivial. Assume the contrary, and consider a minimal normal
subgroup $M$ of $G$ such that $M \leq \frat G$. If $q\in\V G-\V{G/M}$, then $G/M$ has an abelian normal Sylow $q$-subgroup,
but the fact that $\fit{G/M}=\fit G/M$ yields that $G$ has a normal Sylow $q$-subgroup $Q$ as well.
Since $Q$ cannot be abelian, then $Q\cap M\neq 1$, so $M\subseteq Q$; in fact we have $M=Q'$, and we see that $q$ is
uniquely determined by $M$. In other words, $\V G-\V{G/M}=\{q\}$ and, as observed in the first paragraph, $q\neq p$.
Now, by our assumptions, the graph $\Delta(G) - p$  is disconnected, and we denote by $\pi$ the connected component of $q$ in this subgraph.
Let $L$ be a $q$-complement of $G$, and consider any non-principal $\lambda\in\irr{M}$. By Theorem 13.31 of~\cite{Is}, and by
the fact that $M\subseteq\zent{Q}$, there exists $\psi\in\irr{Q|\lambda}$ such that the inertia subgroup
$I_L({\psi})$ coincides with $I_L({\lambda})$; moreover, by coprimality, $\psi$ has an extension to $I_G(\psi)=QI_L(\psi)=QI_L(\lambda)$.
Note that, clearly, $\psi(1)$ is divisible by $q$.
Now,  let  $t\in \V G - \pi$, $t\neq p$. In particular, $q$ and $t$ are not adjacent in $\Delta(G)$ and hence Gallagher's
theorem and Clifford correspondence yield that $I_L(\lambda)$ contains a Sylow $t$-subgroup of $L$ as a normal subgroup
(and this Sylow subgroup is also abelian).
On one hand, this implies that $\cent L M$ (which is contained in $I_L(\lambda)$) has an abelian normal Sylow $t$-subgroup.
On the other hand, \cite[Theorem~2.5]{DKP} ensures that $L/\cent LM$ is a group of semilinear maps on $M$, so it is solvable.
Hence,  $L$, and therefore $G=QL$, is $t$-solvable.  
We conclude that the connected component $\pi$ of $\Delta(G)-p$ contains all the   prime divisors of any non-solvable chief factor of $G$. In particular, it contains a prime different from $q$ and this implies that  $\Delta(G/M)$ is either disconnected or
connected with cut-vertex $p$.
By the minimality of $G$, and since $R(G/M) = R/M$, we have that  $G/R \cong (G/M)/(R/M)$ is an  almost-simple group with
socle $X/R \in \mathcal{L}$ and that $\V{G/M} = \pi(G/R) \cup \{p\}$. 
Writing $\pi(G/R) = \pi(X/R) \cup \sigma$, with $\sigma \cap \pi(X/R) = \emptyset$, by Theorem~\ref{MoretoTiep} we see
that if $\sigma \neq \emptyset$, then $X/R$ is a simple group of Lie type of characteristic (say) $r$ and that all primes
in $\sigma$ are adjacent to every prime in $\pi(X/R)$, except possibly to $r$, in $\Delta(G/R)$ and hence in $\Delta(G)$.
Since, as observed above, every prime divisor of \(|X/R|\) lies in \(\pi\), we deduce  that  $\V{G/R}$ is contained in the connected component $\pi$ of $q$ in  $\Delta(G)-p$ and hence,
recalling that $\V G - \{q\} = \V{G/M} = \pi(G/R) \cup \{p\}$, $\Delta(G) - p$ would be connected, a contradiction.  
Hence, $\frat G=1$.

So, $F = \fitg G = \fit G$ is a direct product \(M_1\times\cdots\times M_k\) of (abelian) minimal normal subgroups of $G$. Since $F=\cent G F=\bigcap_{i=1}^k\cent G{M_i}$, the factor group \(G/F\) embeds in the direct product of the \(G/\cent G{M_i}\), and we deduce that there exists a minimal normal subgroup $M$ of $G$ such that $G/\cent GM$ is non-solvable.
Since $M$ is abelian and $\frat G = 1$, $M$ has a complement $H \cong G/M$ in $G$. Let $\o G = G/\cent HM$. 
Let $\pi_0 = \V G - (\V{G/M} \cup \{p\})$, $q \in \pi_0$ and $Q$ a Sylow $q$-subgroup of $G$.  
As $QM/M$ is an abelian normal subgroup of $G/M$,  we observe that $q$ does not divide $|M|$, as otherwise $M = Q' \leq \frat Q \leq \frat G = 1$,  a contradiction. So, $Q$ is abelian, and we can assume $Q \leq H$; also, we have $[QM, QM] = [M, Q] \nor G$ and
it follows that $[M, Q] = M$. Denoting by \(\widehat{M}\) the dual group of \(M\), we then get $[\widehat{M}, Q] = \widehat{M}$:
in particular, $q$ divides $|H: I_H(\mu)|$ for every non-trivial $\mu \in \widehat{M}$.
Let $s \in \V G$ be a prime that is not adjacent to $q$ in $\Delta(G)$ (recall that, \(q\) being different from \(p\), such an $s$ certainly exists). 
Then, as $\mu$ extends to $I_G(\mu) = M I_H(\mu)$, Gallagher's theorem and Clifford correspondence imply that
$I_H(\mu)$ contains a Sylow $s$-subgroup of $H$ as a normal subgroup, and this Sylow \(s\)-subgroup is also abelian, for every non-trivial $\mu \in \widehat{M}$. Note that \(s\) divides \(|\overline{H}|\), since otherwise a Sylow \(s\)-subgroup of \(H\) would be contained in \(\cent H M\), and it would easily follow that \(G\) has an abelian normal Sylow \(s\)-subgroup, against the fact that \(s\in\V G\).

Let $\pi$ be the connected component of $\Delta(G) - p$ that contains $q$.
Since by assumption $\Delta(G) - p$ is disconnected, the set $\delta = \V G -(\pi \cup \{p\})$ is
non-empty. 
By the previous paragraph, $\pi_0 \subseteq \pi$ and $(\o H, \widehat{M})$ satisfies $\mathcal{N}_s$ for every $s \in \delta$.
As $\delta \neq \emptyset$, Proposition~\ref{P3.12+} yields that the solvable radical $\o Y$ of $\o H$ is cyclic.
Let $\o X /\o Y$ be a minimal normal subgroup of $\o H / \o Y$; then $\o X / \o Y \cong S^k$, where $S$ is a non-abelian simple group. 
If $\pi(S) - \{p\} \subseteq \delta$, then (since $\o H$ has abelian Hall $\delta$-subgroups) $S$   has an abelian
$p$-complement, so there exists a non-trivial element $x$ of $S$ such that $|S:\cent Sx|$ is a power of the prime $p$, and this
is not possible by Burnside's theorem (\cite[Theorem 3.9]{Is}).
So, the connected component $\pi$ of $\Delta(G) -p$ intersects non-trivially $\pi(S)$ and, since $\pi(S) \subseteq \V{G/M}$,
it follows that $\Delta(G/M)$ is either disconnected or connected with cut-vertex $p$. By the minimality of $G$,
we have that $H \cong G/M$ has a unique non-solvable composition factor $S \in \mathcal{L}$, and that $\Delta(S)$ is either disconnected or connected with cut-vertex $p$. 
Hence, $\o X/\o Y \cong S$; observe also that $\Delta(S)$ is a subgraph of $\Delta(G)$.
As $\o Y$ is cyclic, $\o Y \leq \zent{\o X}$ and then $\o K = \o{X}'$ is a quasi-simple group (\cite[(31.1)]{A}). Now, if \(q\) divides $|\overline{K}|$, then ($\overline{K}$ being perfect) it divides $|\overline{K}/\zent{\overline{K}}|$ as well, hence it divides $|S|$ and it is therefore a vertex of \(\Delta(G/M)\), not our case.
Thus, recalling that $\o Q$ is a normal subgroup of $\o H$, we deduce that $\o Q$ centralizes $\o K\trianglelefteq\o H$;
moreover, we have that $(\o K, \widehat{M})$ satisfies $\mathcal{N}_s$ for every $s \in \delta \cap \pi(S)$. 
But we know that $S$ lies in $\mathcal{L}$, whence Lemma~\ref{Nqvari} yields that either $S \cong {\rm PSL}_2(t^a)$ or $S \cong {\rm Sz}(2^a)$.

In the former case, we get either $\o K = S$ or  $\o K \cong {\rm SL}_2(t^a)$. Recall that $(\overline{K},\widehat{M})$ satisfies \(\mathcal{N}_s\) for every \(s\in\delta\): so,  by Lemma~\ref{SL2Nq} and Remark~\ref{remPSL}, as $\widehat{M}$ is a faithful $\o K$-module, either  ($s = t$ and)  $|\widehat{M}| = t^{2a}$, or ($s= 3$ and) $|\widehat{M}| = 3^4$ with $\o K \cong {\rm SL}_2(5)$, or finally ($s= 3$ and) $|\widehat{M}| = 3^6$ with $\o K \cong {\rm SL}_2(13)$.
If $|\widehat{M}| = t^{2a}$, then we can see $\o H$ as a subgroup of ${\rm GL}_{2a}(t)$ and, 
by Lemma~\ref{clg}, $|\cent {\o H}{\o K}|$ divides  $|\o K|$.  Since $\o Q \leq \cent{\o H}{\o K}$ has order
coprime to $|\o K|$, we get a contradiction as  $\o Q \cong Q \neq 1$.
If $|\widehat{M}| = 3^4$, then $\o K \cong {\rm SL}_2(5)$ can be seen as a subgroup of $\o L = {\rm GL}_4(3)$; one computes that
then $|\cent{\o L}{\o K}| = 8$, which is not possible since $\o Q \leq \cent{\o L}{\o K}$.
If $|\widehat{M}| = 3^6$, then we see $\o K \cong {\rm SL}_2(13)$  as a subgroup of $\o L = {\rm GL}_6(3)$; one computes that
$\cent {\o L}{\o K} = \zent{\o L}$, of order $2$, and this is again impossible as $\o Q \leq \cent{\o L}{\o K}$.

On the other hand, if $S \cong {\rm Sz}(2^a)$, then by the structure of the degree graph of the Suzuki groups and the fact that $\Delta(S)$
is connected with cut-vertex $p$, we deduce that $2^a-1 = p$ and that $\delta \cap \pi(S)$ is either $\{ 2\}$ or
$\pi(2^{2a}+1)$. The first possibility is ruled out by Lemma~\ref{Nqvari}
and the second one is excluded because $S$ does not have abelian $\pi(2^{2a}+1)$-Hall subgroups. This is the final contradiction that completes the proof.
\end{proof}
%

%
\section{The main result}\label{proofmain}
We are now ready to prove the main result of this paper.

\begin{theorem}\label{main}
  Let $G$ be a non-solvable group with no composition factors isomorphic to ${\rm PSL}_2(t^a)$ ($t$ prime, $t^a\geq 4$),
  let $R$ and $K$ be, respectively, the solvable radical and the solvable residual of $G$, and let \(p\) be a prime number.
  Then the graph \(\Delta(G)\) is connected and has cut-vertex $p$ if and only if 
one of the following holds.  
\begin{enumeratei}
\item $K \cong {\rm{Sz}}(2^a)$,  $a\geq 3$ is a prime, \(p=2^a-1\)  and $\V{G/K}\sbs \{p\}$.
\item  $K \cong {\rm{PSL}}_3(4)$,   $|G:KR|\in \{1,3\}$, \(p=5\) and $\V{G/K}\sbs \{5\}$.
\item $K \cong {\rm{M}}_{11}$,  $G=K\times R$, \(p=5\) and $\V R\sbs \{5\}$.
\item $K \cong {\rm{J}}_1$,  $G=K\times R$, \(p=2\) and $\V R\sbs \{ 2\}$.
\end{enumeratei}
Moreover,  $p$ is the unique cut-vertex of $\Delta(G)$  and $\Delta(G) - p$ has two complete connected components.
In all cases except case {\rm{(d)}}, one of the connected components of $\Delta(G) - p$ has size $1$ and, except case {\rm{(d)}} when
$R$ is abelian, $p$ is a complete vertex of $\Delta(G)$. 
\end{theorem}
\begin{proof}  
  By Theorem~\ref{0.2} we know that $G/R$ is an almost-simple group with socle $X/R$ and,
  as $X/R \not\cong {\rm{PSL}}_2(t^a)$ by our assumptions, $X/R$ is isomorphic to a simple group in the following list:
  $$\mathcal{L}_0 =  \{ {\rm{Sz}}(2^a)\; (a \geq 3, \;a \text{ odd}), \;{\rm{PSL}}_3(4),\; {\rm{M}}_{11},\; {\rm{J}}_1\}.$$
  Moreover, since the degree graphs of the groups in $\mathcal{L}_0$ are connected (see Theorem~\ref{CDsimple}), by Theorem~\ref{0.2} we also have that
$p$ is a cut-vertex of  $\Delta(X/R)$.  

  We observe that $KR = X$, because  $KR$ certainly contains $X$, and $X$ contains $KR$, as $K/K\cap X \cong KX/X$
  is solvable, thus \(X\supseteq K\). Now, setting $L = K \cap R$, we see that   $K/L \cong X/R$.
  We will first show that $L = 1$: working by contradiction, we assume that $L \neq 1$.
Let $L/L_0$ be a chief factor of $K$ (possibly $L_0=1$) and consider a non-principal character $\lambda\in \irr {L/L_0}$.
If $\lambda$ is $K$-invariant, then $\ker{\lambda}/L_0$ is a proper $K$-invariant subgroup of $L/L_0$.
Hence,  $\ker{\lambda} = L_0$ and  $\lambda$ is a faithful and $K$-invariant linear character of $L/L_0$. It follows that
$L/L_0$ is cyclic and $L/L_0 \leq \zent{K/L_0}$. Since $K/L_0$ is a perfect group, by~Theorem 11.19 of \cite{Is} it follows
that $L/L_0$ is isomorphic to a subgroup of the Schur multiplier ${\rm{M}}(K/L)$.
In particular, we deduce that no non-principal $\lambda \in \irr{L/L_0}$ can be $K$-invariant if ${\rm{M}}(K/L) = 1$. 
We will now derive a contradiction, by separately considering the various possibilities for the simple group $K/L$.

  \medskip
  \noindent
  $K/L \cong {\rm{Sz}}(q)$ (where  $q = 2^a$ and  $a \geq 3$ is an odd integer). 

  By Lemma~\ref{Sz} and Theorem~\ref{CDsimple}, we know that
  $2^a-1 = p$ and that  $2$ is only adjacent to $p$ in $\Delta(K/L)$.
  Let $\lambda \in \irr{L/L_0}$.
  Assuming that $\lambda$ is $K$-invariant, since the Schur multiplier of $\mathrm{Sz}(2^a)$ is trivial for $a \geq 5$ (\cite[Theorem 1]{AG}), then
$a=3$, $p = 7$ and ${\rm{M}}(K/L) \cong C_2\times C_2$ (\cite[Theorem 2]{AG}).
As one can check in~\cite{atlas},  a perfect central extension of $\rm{Sz}(8)$ by a
subgroup of order $2$ has irreducible characters of degree $40$, and this is a contradiction.
Therefore, $I_K(\lambda)/L$ is a proper subgroup of $K/L$, and hence it is contained in a maximal subgroup,  of $K/L$.
  We recall that the maximal subgroups of
  ${\rm{Sz}}(q)$ are isomorphic to one of the following groups (see, for instance, \cite[Theorem 4.1]{Wilson}).
\begin{enumeratei}
\item Frobenius groups $Q\rtimes C_{q-1}$ where $Q\in \syl 2 {\text{Sz}(q)}$;
\item Suzuki simple groups ${\rm{Sz}}(q_0)$, where $q= q_0^r$, $r$ is a prime and $q_0 \geq 8$;
\item Dihedral groups $D_{2(q-1)}$;
\item $C_{f_{+}}\rtimes C_4$, where $f_{+} = q +\sqrt{2q}+1$.
\item $C_{f_{-}}\rtimes C_4$, where $f_{-} = q -\sqrt{2q}+1$.
\end{enumeratei}
Since $2$ is only adjacent to $p$ in $\Delta(K/L)$, by Clifford correspondence the index $|K:I_K(\lambda)|$ must be either odd
or coprime to $q^2+1$, and hence the unique possibility for  $I_K(\lambda)/L$ is to be contained
in a Frobenius group $Q \rtimes C_{q-1}$ (case (a) above) and that $I_K(\lambda)/L$ contains a Sylow $2$-subgroup $Q$ of $K/L$.
This is true for every non-trivial $\lambda$ in the dual group $V$ of $L/L_0$, and hence  $(K/L, V)$ verifies the condition 
$\mathcal{N}_2$; but this is impossible by Lemma~\ref{Nqvari}.

    \medskip
  \noindent
  $K/L \cong \rm{PSL}_3(4)$.
  
  Recall that \(\Delta(K/L)\), which is a subgraph of \(\Delta(G)\), is a connected graph whose vertex set is \(\{2,3,5,7\}\) and it has cut-vertex \(5\) (See Figure~1); moreover, in \(\Delta(K/L)\), the vertex \(5\) is the only neighbour of \(2\). Since we have $\V G = \pi(K/L)\cup\{p\}$ by Theorem~\ref{0.2}, the disconnectedness of \(\Delta(G)-p\) forces $p=5$. Hence, $2$ is adjacent only to $5$ in $\Delta(G)$ as well.
  Assume first that there exists a non-principal $\lambda \in \irr{L/L_0}$ which is $K$-invariant. 
 As the Schur multiplier of $\text{PSL}_3(4)$ is isomorphic to $C_4\times C_4 \times C_3$ (\cite{atlas}), by the above discussion
 we deduce that $L/L_0$ is a cyclic central subgroup of $K/L_0$,  and that $|L/L_0|$ is either  $2$ or $3$.
 One can check in~\cite{atlas} that in the first case $K/L_0$ has an
 irreducible character of degree $28$, while in the second case $K/L_0$ has an irreducible character of degree $84$, giving a contradiction in both cases. 
So we may assume that $I_K(\lambda)<K$, and hence  $I_K(\lambda)/L$ is contained in a maximal subgroup $H/L$ of $K/L$.
The maximal subgroups of $\text{PSL}_3(4)$ are isomorphic to $(C_2)^4\rtimes A_5$, $A_6$, $\text{PSL}_3(2)$, and $(C_3)^2 \rtimes Q_8$ (\cite{atlas}) and, since $2$ is only adjacent to $5$ in $\Delta(G)$, it follows that
$H/L$ is isomorphic to $(C_2)^4\rtimes A_5$. Moreover, $I_K(\lambda)/L$ is forced to contain a Sylow \(2\)-subgroup of \(K/L\): in other words, the action of \(K/L\) on the dual group of \(L/L_0\) satisfies condition \(\mathcal{N}_2\). However, this is impossible by Lemma~\ref{Nqvari}, and the desired conclusion is achieved in this case.

    \medskip
  \noindent
  $K/L \cong \rm{M}_{11}$.
  
By Theorem \ref{0.2} and  the structure of  $\Delta(\rm{M}_{11})$ (see Figure~1), we know  that $p = 5$ is a cut-vertex of both
$\Delta(G)$ and  $\Delta(K/L)$,  that $\V G = \{2,3,5,11 \}$ and that $3$ is adjacent only to $5$ in $\Delta(G)$.
As $K/L \cong M_{11}$ has trivial Schur multiplier (\cite{atlas}), then $K/L$ acts faithfully and irreducibly on the dual group $V$ of $L/L_0$ and,
since $K/L$ cannot have any regular orbit on $V$ (as otherwise $\Delta(G)$ would be a complete graph)
by \cite[Theorem~2.3]{KP} and \cite[Theorem 1.1]{FMOW} we deduce that $V$ is either a $10$-dimensional $\mathbb{F}_q[K/L]$-module, with $q \in \{2, 3\}$,
or $V$ is a $5$-dimensional $\mathbb{F}_3[K/L]$-module.
In the first case, one computes (via GAP \cite{GAP}) that there is a $K/L$-orbit in $V$ whose length is divisible by all the prime divisors of $K/L$, so by Clifford theory $\Delta(G)$ is a complete graph, a contradiction.
On the other hand, there are two isomorphism types of $5$-dimensional irreducible  $\mathbb{F}_3[K/L]$-modules.
One of them has a $K/L$-orbit of size $132 = 2^2\cdot3\cdot 11$, so by Clifford theory $3$ and $11$ would be adjacent in $\Delta(G)$, a contradiction. The second one, which we keep denoting as  $V$,  has non-trivial orbit sizes $22$ and $220$, is self-dual  and the second cohomology group $H^2(K/L, V)$ is trivial.
One checks again via GAP \cite{GAP} that the semidirect product $H =  V \rtimes K/L$ has an irreducible character of degree $660 = 2^2\cdot 3 \cdot 5 \cdot 11$ and, since $H$ is isomorphic to a normal section of $G$, it follows that $\Delta(G)$ is a complete graph, a contradiction.

    \medskip
  \noindent
  $K/L \cong {\rm J}_1$.
  
By Theorem \ref{0.2} and  the structure of  $\Delta(\rm{J}_{1})$ (see Figure~1), we know  that $p = 2$ is a cut-vertex of both
$\Delta(G)$ and  $\Delta(K/L)$, and  that $\V G = \pi(K/L) = \{2,3,5,7,11, 19 \}$.
As $K/L \cong {\rm{J}}_{1}$ has trivial Schur multiplier (\cite{atlas}), then $K/L$ acts faithfully and irreducibly on the dual
group $V$ of $L/L_0$.
Again Theorem~2.3 of \cite{KP} and Theorem~1.1 of \cite{FMOW} ensure that either $K/L$ has a regular orbit on $V$ or  
 $V$ is the $20$-dimensional $\mathbb{F}_2[K/L]$-module.
One checks via GAP \cite{GAP} that in the last case  there is a $K/L$-orbit in $V$ of size $2^2\cdot 3 \cdot5 \cdot 7 \cdot 11 \cdot 19$. Hence, in any case, by Clifford theory  $\Delta(G)$ is a complete graph, a contradiction. 

\medskip
Our conclusion so far is that $L = 1$, and hence that  $K$ is a minimal normal subgroup of $G$. 
We will conclude the proof of the claims (a)--(d)  by again  considering separately the possibilities for
$K \in \mathcal{L}_0$.
By Theorem~\ref{CDsimple} $\Delta(K)$ is connected and,  by Theorem~\ref{0.2}, $p$ is a cut-vertex of both $\Delta(G)$ and  
$\Delta(K)$. 
Let
$C = \cent GK$. As $R \cap K = 1$, we have $R \leq C$ and, since 
$G/R$ is almost simple,
we deduce that $R = C$.

If $K \cong {\rm{Sz}}(2^a)$ where  $a\neq 1$ is an odd positive integer then, by Lemma~\ref{Sz},
 we know that the vertex $2$ is only adjacent in $\Delta(K)$ to the prime divisors of $2^a-1$. 
 It follows that $2^a-1$ is a power of the prime $p$, so $2^a - 1 = p$ is a Mersenne prime by Proposition 3.1 of \cite{MW}, and
 hence $a$ is a prime number. 
 Let $\theta$ be the Steinberg character of
 the socle $X/R \cong K$ of $G/R$; so $\theta(1)$ is a power of $2$ and $\theta$ has an extension $\psi$
 to $G/R$. Seeing $\psi$ as an irreducible character of $G$ by inflation, then $\psi$ is an extension of the character
 $\theta_0\in \irr K$ that corresponds to $\theta$ in the natural isomorphism between $X/R$ and $K$. 
 By Gallagher's theorem we get  $\irr{G\,|\, \theta_0}= \{\psi \gamma \,|\,  \gamma\in \irr {G/K}\} $, and hence $2$ is
 adjacent in $\Delta(G)$ to every prime in $\V{G/K}$; this forces  $\V{G/K}\sbs \{p\}$ and (a) is proved. 

 If $K \cong {\rm PSL}_3(4)$ then, by Theorem~\ref{0.2},  $\V G = \pi(K) = \{2, 3, 5, 7\}$ and, as $p$ is a cut-vertex of both $\Delta(G)$ and $\Delta(K)$, then $p=5$ (see Figure~1) and $\Delta(G) = \Delta(K)$. 
 The outer automorphism group of $\rm{PSL}_3(4)$ is isomorphic to $C_2\times S_3$ and all almost-simple extensions of
 $\rm{PSL}_3(4)$ by a subgroup of order $2$ have an irreducible character of degree divisible by $6$ (\cite{atlas}). 
It follows that $|G: KR|$ is odd, and hence  $|G:KR|\in\{1,3\}$.
As in the previous paragraph, the Steinberg character $\theta_0$ of $K$ has an extension $\psi$ to $G$,
and then  by Gallagher's theorem $\irr{G\,|\,\theta_0}= \{\psi \gamma \,|\,  \gamma\in \irr {G/K}\} $.
Hence $2$ is adjacent to every prime in $\V{G/K}$, which forces $\V{G/K}\sbs \{5\}$ and 
(b) is proved. 

If $K \cong \rm{M}_{11}$ or $K \cong \rm{J}_{1}$,  then $G = K \times R$ because the outer automorphism group of $K$ is trivial. Hence, $\Delta(G) = \Delta(K) \ast \Delta(R)$ is the join of the degree graphs of $K$ and $R$ and, since $G$ has cut-vertex $p$,
recalling the graphs of $\rm{M}_{11}$ and $\rm{J}_1$ (see Figure~1) it follows that $\V R \sbs \{ p \}$.
If $K \cong \rm{M}_{11}$,  then $p =5$ is a complete vertex of $\Delta(G) = \Delta(K)$. If $K \cong \rm{J}_1$, then
$p =2$ and either  $\Delta(G) = \Delta(K)$, or $\Delta(G)$ is obtained from $\Delta(K)$ by adding an edge between $2$ and $11$.
So, also (c) and (d) are proved and the proof of the `only if' part of the statement is complete.  

\medskip
\noindent
Conversely, assume that 
one of the conditions
(a), (b), (c) or (d) holds. Since $K$ is a non-abelian simple group and $G/K$ is solvable, $\cent GK$ (which is isomorphic to a subgroup of $G/K$) is solvable as well and hence,
as the solvable radical $R$ of $G$ is contained in $\cent GK$, we see that $R = \cent GK$. Morevoer, $G/KR$ is isomorphic to a
subgroup of the outer automorphism group of $K$. 

Assume first (a), so  $K \cong {\rm{Sz}}(2^a)$ where $p = 2^a-1$ is a prime, $a \geq 3$ is a prime  and $\V{G/K}\sbs \{p\}$.
The outer automorphism group of a Suzuki simple group consists only of field automorphisms (\cite{atlas}), so  $G/KR$ is either trivial or of order $a$. Recalling that \(p\) is a complete vertex of \(\Delta(K)\), if $G = K \times R$  then $\Delta(G) = \Delta(K)$ and $p$ is both a complete vertex and the unique cut-vertex  of $\Delta(G)$. 
So, we can assume that $|G/KR| = a$, and by Theorem~1.1 of \cite{G} we have that $a$ is adjacent in $\Delta(G/R)$ and hence in
$\Delta(G)$ to every other vertex in  
$\pi(K) - \{ 2 \}$ (by the way, one can easily check that $a \in \pi(K)$ only when $a = 5$).  Hence, $\Delta(G) - 2$ is a complete graph. 
If $\chi \in \irr G$ has even order, then $\chi$ lies above the Steinberg character $\theta$ of $K$, because $a$ is odd and
$\V{R} \sbs \V{G/K}\sbs \{ p \}$.
Since, by the same argument used above, $\theta$ is $G$-invariant and extends to $G$, by Gallagher's theorem
$\chi(1) = \theta(1) \beta(1)$ for some $\beta \in \irr{G/K}$.  
Hence, $2$ is adjacent only to $p$ in $\Delta(G)$,   and therefore $p$ is a cut-vertex of $\Delta(G)$. Moreover, also in this case \(p\) is a complete vertex of \(\Delta(G)\).

Assume now (b), so $K \cong {\rm{PSL}}_3(4)$, $|G:KR|\in \{1,3\}$ and $\V{G/K}\sbs \{5\}$.
If $G = K \times R$, then $\Delta(G) = \Delta(K)$ as $\V R \sbs \{ 5\}$ and $5$ is a complete vertex of $\Delta(K)$, thus
we are done.
So, we can assume $|G: KR|=3$. If $\chi\in \irr{G}$ has even degree, then as in the previous paragraph we deduce that
$\chi$ lies over the Steinberg character $\theta$ of $K$ and that $\chi(1) = \theta(1)\beta(1)$ for some $\beta \in \irr{G/K}$. 
Therefore $2$ is only adjacent to $5$ in $\Delta(G)$,
and hence again $\Delta(G) = \Delta(K)$. So, $5$ is a both a complete vertex and the only cut-vertex  of $\Delta(G)$. 

Similarly, assuming  (c) (i.e.  $K \cong {\rm{M}}_{11}$),  $G=K\times R$  and $\V R\sbs \{5\}$, we have that
that $\Delta(G) = \Delta(K)$ and all assertions follow.

If we have (d), so $K \cong {\rm{J}}_1$,  $G=K\times R$ and  $\V R\sbs \{ 2\}$, then either
$\Delta(G) = \Delta(K)$ or $\Delta(G)$ is obtained from $\Delta(K)$ by adding the edge $\{2, 11\}$ (so $2$ is a complete vertex of $\Delta(G)$ in this case). It follows that $2$ is the only cut-vertex of $\Delta(G)$, and it is a complete vertex if and only if \(R\) is non-abelian.

\smallskip
Finally, the degree graphs of the relevant groups are displayed in Table~1 (cases (a--d)), and all the remaining assertions of the statement follow at once.
\end{proof}

\section{A proof of Theorem C}\label{proofC}

In this concluding section we prove Theorem~C, that gives a characterization of the non-solvable groups whose degree graph is disconnected and has a cut-vertex. 

We recall first some features of the degree graphs of the almost-simple groups with socle isomorphic to $\PSL{t^a}$.

\begin{lemma}\label{W}
  Let $G$ be an almost-simple group with socle $S \cong \PSL{t^a}$, where \(t\) is a prime and $t^a \geq 4$.
  \begin{enumeratei}
  \item   If \(q\) is a prime divisor of $|G/S|$, then \(q\) is adjacent in \(\Delta(G)\) to every prime in $\pi(S) - \{ t\}$.
  \item Assume $t^a > 5$.  If $t$ divides $|G:S|$ then $\Delta(G)$ is connected, while if $t$ does not divide $|G:S|$, then $t$ is an isolated vertex of $\Delta(G)$. 
\end{enumeratei}
\end{lemma}
\begin{proof}
 Claim (a) follows from Theorem~A of~\cite{W1}, and (b) is Theorem~2.7 of~\cite{LW}.
\end{proof}

\begin{proof}[Proof of Theorem C]
  Let us first assume that $\Delta(G)$ is a disconnected graph and that it has a  cut-vertex $p$.
  We recall that, by the main theorem of \cite{MSW}, $\Delta(G)$ has at most three connected components. If $\Delta(G)$ has three connected components then, by Theorem~\ref{LewisWhite},  $G=\text{SL}_2(2^a) \times A$, where $A$ is an
  abelian group and $a$ is an integer larger than \(1\). In this situation, the three connected components of $\Delta(G)$ are complete
  graphs with vertex sets $\{2\}$, $\pi(2^a+1)$ and $\pi(2^a-1)$, and this is against the existence of a cut-vertex in $\Delta(G)$.
  Thus, we deduce that $\Delta(G)$ has two connected components.  By Theorem~\ref{LewisWhite}, $G/K$ is an  abelian group and  there
  is a   normal subgroup $N \leq  K$ of $G$  such that  $K/N\cong \PSL{t^a}$ with $t^a\geq 4$; by Remark~\ref{verticesdisc}, $\V G = \pi(G/R)$.
  Moreover,  if $N>1$, then  either $K\cong \SL{t^a}$ or there exists a minimal  normal subgroup $L$ of $G$, $L \leq N$,  such
  that $K/L\cong \SL{t^a}$ and $L$ is isomorphic to the natural module for $K/L$.
  In the last case, as the natural module for $\SL{t^a}$ is self-dual, the stabilizer of any non-trivial irreducible character of $L$  is a Sylow $t$-subgroup of $K$,   and hence  Clifford theory and Lemma~\ref{W} yield that $\Delta(G) - t$ is a
  complete graph, which is again impossible by the assumption of the existence of a cut-vertex in $\Delta(G)$.
  Hence, either $N= 1$ and  $K \cong {\rm PSL}_2(t^a)$ or  $|N| = 2$ and  $K\cong \SL{t^a}$.
  Let now $C$ be a subgroup of $G$ such that $C/N = \cent{G/N}{K/N}$. Clearly, $R \subseteq C$; more precisely, $R \subseteq \cent GK$, because
  $[K, R] \subseteq N$, so $[K, R, K] = 1 = [R,K, K]$ and hence, recalling that $K$ is a perfect group, by the Three Subgroups lemma,
  $[K, R] = [K, K, R] = 1$.  Conversely, as $K/N$ is the only non-solvable composition factor of $G$, $C$ is a solvable normal subgroup
  of $G$, so $C \subseteq R$ and hence $C = R$.
  Therefore, Theorem~\ref{LewisWhite} yields that if $t^a > 5$, then  $t$ does not divide $|G:KR|$.
If $t \neq 2$, then by  Lemma~\ref{PSL2} and Lemma~\ref{W} the assumption that 
$p$ is a cut-vertex of $\Delta(G)$ yields that  $|G:KR|$ is necessarily a power of $p$ and that  $p=2$.
Moreover, again by the assumption that $p = 2$ is a cut-vertex of $\Delta(G)$,  none of $t^a+1$ and $t^a-1$ can be a power of $2$, and hence by~\cite[Proposition 3.1]{MW}  $t^a \neq 9$ and $t^a$  is neither a Fermat nor a Mersenne prime. (In particular, $t^a \neq 5$.) 
If $t=2$ and $t^a \neq 4$, then  Lemma~\ref{PSL2} implies that $G \neq KR$ 
and, together with Lemma~\ref{W}, that  $|G:KR|\neq 1$ is a power of $p\neq 2$.  
Finally, if $t^a= 4$, then $K$ is isomorphic to  $\SL{4}$ and
$\V G = \pi(G/R) = \{2,3,5\}$ and, since $\Delta(G)$ is disconnected, $\Delta(G)$ cannot have any cut-vertex.  

\smallskip
Conversely, we assume that  either  $K \cong \PSL{t^a}$ or
$K \cong \SL{t^a}$ (where \(t\) is a prime and \(t^a\geq 4\)),  $G/K$ is abelian,  $|G:KR| = p^b$ (with $b \geq 0$), and that either (a) or (b) holds. 
Suppose first that $t=2$. Let $\chi$ be an irreducible character of $G$ such that $\chi(1)$ is even and let $\psi$ be an
irreducible constituent of $\chi_{K \times R}$. Then $\psi = \alpha \times \beta$, where $\alpha \in \irr K$  and $\beta \in \irr R$; note that $\psi(1)$ is even as $|G:KR|$ is odd. Since $R$ (which is isomorphic to a subgroup of $G/K$) is abelian, we see that  $\alpha(1)$ is even and
hence $\alpha$ is the Steinberg character of $K$; in particular, $\alpha$ is $G$-invariant and  $\alpha(1) = 2^a$. 
As $G/K$ is abelian, also $\beta$ is $G$-invariant  and hence $\psi$ is $G$-invariant.
Now, $G/KR$ is cyclic because it is isomorphic to a subgroup of the outer automorphism group of $K$;  hence, by Corollary 11.22 of~\cite{Is}, $\psi$ extends to $G$ and,  by Gallagher's theorem,   $\chi(1)$ is a power of $2$. We deduce that $2$ is an isolated vertex of $\Delta(G)$ and, in particular, $\Delta(G)$ is  disconnected.
Moreover, by part (a) of Lemma~\ref{W} the vertex $p$ is adjacent in $\Delta(G)$ to all the vertices in $\V G - \{2, p\}$, so $p$ is a
cut-vertex of $\Delta(G)$.

Finally, we assume that $t$ is odd. We observe that, by the same argument we used above, $R/N = \cent {G/N}{K/N}$; hence, by Theorem~\ref{LewisWhite}, $\Delta(G)$ is disconnected.
Moreover, by Lemma~\ref{PSL2} and the fact that  $\Delta(K)$ is a subgraph of $\Delta(G)$,   $p=2$ is adjacent in $\Delta(G)$ to all
vertices of $\V G - \{ 2, t\}$. 
Let us now assume, working by contradiction,  that there is a character  $\chi \in \irr G$ such that $qs$ divides $\chi(1)$
for some odd primes $q\in \pi(t^a-1)$ and $s\in \pi(t^a+1)$. Let $\theta\in \irr{KR}$ be an irreducible constituent of $\chi_{KR}$.
Since $|G:M|$ is power of $2$, then $qs$ necessarily divides $\theta(1)$.
Observe that $N$, being a cyclic group of order at most $2$,  is central in $G$ and hence $KR = KD \times H$, where
$D\in \syl 2 R$ and $H$ is the abelian $2$-complement of $R$.
It follows that $qs$ divides the degree of any irreducible constituent $\alpha$ of $\theta_K$, and this is a contradiction by
Lemma~\ref{PSL2}. 
Hence, no vertex of $\pi(t^a+1) - \{ 2\}$ is adjacent to any vertex of $\pi(t^a+1) - \{ 2\}$ and we conclude that $\Delta(G) = \Delta(K)$.
Finally, by (a) both  $\pi(t^a-1) -\{2\}$ and $\pi(t^a+1) -\{2\}$ are  non-empty sets of vertices of $\Delta(G)$  and
hence $2$ is a cut-vertex of $\Delta(G)$. The proof is complete.  
\end{proof}

\end{document}